\documentclass[sigconf]{acmart}
\renewcommand\footnotetextcopyrightpermission[1]{} 
\usepackage{booktabs} 

\usepackage{algorithm} 
\usepackage{algorithmic} 
\usepackage{mdwlist} 
\usepackage[mathscr]{euscript} 
\usepackage{amsbsy} 
\usepackage{amsfonts} 
\usepackage{subfig} 
\usepackage{array} 
\usepackage{ragged2e} 

\usepackage{graphicx}
\usepackage{amsmath} 
\usepackage{hyperref} 
\usepackage{multirow} 
\usepackage{footnotebackref}[2012/07/01]
\usepackage{tablefootnote}[2014/01/26] 
\usepackage{color}
\usepackage{xspace}

\usepackage{url}

\setcopyright{none}

\newcommand{\T}[1]{\boldsymbol{\mathscr{#1}}}

\newcommand{\bit}{\begin{itemize*}}
	\newcommand{\eit}{\end{itemize*}}

\newcommand{\hide}[1]{}

\newcolumntype{C}[1]{>{\centering\arraybackslash}p{#1}}
\newcolumntype{R}[1]{>{\RaggedLeft\arraybackslash}p{#1}}
\newcolumntype{L}[1]{>{\RaggedRight\arraybackslash}p{#1}}

\newcommand{\method}{{\sc $S^3$CMTF}\xspace}

\newcommand{\hogwild}{{\sc HOGWILD!}\xspace}

\newcommand{\algorithmicdoinparallel}{\textbf{do in parallel}}
\makeatletter
\AtBeginEnvironment{algorithmic}{%
	\newcommand{\FORP}[2][default]{\ALC@it\algorithmicfor\ #2\ %
		\algorithmicdoinparallel\ALC@com{#1}\begin{ALC@for}}%
	}
	\makeatother

\begin{document}
\title{Fast, Accurate, and Scalable Method for\\ Sparse Coupled Matrix-Tensor Factorization}

\author{Dongjin Choi}
\affiliation{%
	\institution{Seoul National University}
}
\email{skywalker5@snu.ac.kr}

\author{Jun-Gi Jang}
\affiliation{%
	\institution{Seoul National University}
}
\email{elnino9158@gmail.com}

\author{U Kang}
\affiliation{%
	\institution{Seoul National University}
}
\email{ukang@snu.ac.kr}

\begin{abstract}
    How can we capture hidden properties from a tensor and a matrix data simultaneously in a fast, accurate, and scalable way? 
Coupled matrix-tensor factorization (CMTF) is a major tool to extract latent factors from a tensor and matrices at once.
Designing an accurate and efficient CMTF method has become more crucial as the size and dimension of real-world data are growing explosively. However, existing methods for CMTF suffer from
lack of accuracy, slow running time, and limited scalability. 

In this paper, we propose \method, a fast, accurate, and scalable CMTF method. \method achieves high speed by exploiting the sparsity of real-world tensors, and high accuracy by capturing inter-relations between factors. Also, \method accomplishes additional speed-up by lock-free parallel SGD update for multi-core shared memory systems.
We present two methods, \method-naive and \method-opt. \method-naive is a basic version of \method, and \method-opt improves its speed by exploiting intermediate data. We theoretically and empirically show that \method is the fastest, outperforming existing methods.
Experimental results show that \method is 11$\sim$43$\times$ faster and 2.1$\sim$4.1$\times$ more accurate than existing methods. \method shows linear scalability on the number of data entries and the number of cores. 
In addition, we apply \method to Yelp recommendation tensor data coupled with 3 additional matrices to discover interesting patterns. 
\end{abstract}

%
%
\maketitle
\vspace*{-0.3cm}
\section{Introduction}
    \label{sec:intro}
    Given a tensor data, and related matrix data, how can we analyze them efficiently?
Tensors (i.e., multi-dimensional arrays) and matrices are natural representations for
various real world high-order data.
For instance, an online review site Yelp provides rich information about users (name, friends, reviews, etc.), or about businesses (name, city, Wi-Fi, etc.). One popular representation of such data includes a 3-way rating tensor with (user ID, business ID, time) triplets and an additional friendship matrix with (user ID, user ID) pairs.
Coupled matrix-tensor factorization (CMTF) is an effective tool for joint analysis of coupled matrices and a tensor. The main purpose of CMTF is to integrate matrix factorization \cite{koren2009matrix} and tensor factorization \cite{kolda2009tensor} to efficiently extract the factor matrices of each mode.
The extracted factors have many useful applications such as latent semantic analysis \cite{ding2008equivalence, peng2011equivalence,xu2003document}, recommendation systems \cite{karatzoglou2010multiverse,rendle2010pairwise}, 
network traffic analysis~\cite{Sael201582}, 
and completion of missing values \cite{acar2011all, acar2013understanding, narita2012tensor}.

However, existing CMTF methods do not provide good performance in terms of time, accuracy, and scalability. CMTF-Tucker-ALS \cite{ozcaglar2012algorithmic}, a method based on Tucker decomposition \cite{de2000best}, has a limitation that it is only applicable for dense data.
For sparse real-world data, it assumes empty entries as zero and outputs highly skewed results which are impractical.
Moreover, CMTF-Tucker-ALS does not scale to large data because it suffers from high memory requirement by \textit{M-bottleneck problem} \cite{oh2017s} (see Section \ref{subsec:CMTF} for details).
CMTF-OPT \cite{acar2011all} is a CMTF method based on CANDECOMP/PARAFAC (CP) decomposition \cite{kolda2009tensor}. It has a limitation that it does not take advantage of all inter-relations between related factors because CP decomposition model represents a specific case of the Tucker model in which each factor is related to only a few number of other factors. Therefore, CMTF-OPT undergoes a low model capacity and results in high test error.

In this paper, we propose \method (\textbf{S}parse, lock-free \textbf{S}GD based, and \textbf{S}calable CMTF),
a fast, accurate, and scalable CMTF method which resolves the problems of previous methods.
\method performs parallel stochastic gradient descent (SGD) update, thereby providing much better time complexity than previous methods.
\method has two versions: a basic implementation \method-naive, and an improved version \method-opt which exploits intermediate data for efficient computation. Table \ref{tab:salesman} shows the comparison of \method and other existing methods. The main contributions of our study are as follows:
\begin{table}[tbp]
	\small
	\setlength{\tabcolsep}{1pt}
	\caption{Comparison of our proposed $\mathbf{S^3}$CMTF and the existing CMTF methods. $\mathbf{S^3}$CMTF outperforms all other methods in terms of time, accuracy, scalability, memory usage, and parallelizability.}
	\begin{center}
		{
			\begin{tabular}{L{2.3cm} L{0.8cm} L{1.25cm} L{1.45cm} L{1.1cm} L{1.1cm} }
				\toprule
				\textbf{Method} & \textbf{Time} & \textbf{Accuracy} & \textbf{Scalability} & \textbf{Memory} & \textbf{Parallel} \\
				\midrule
				CMTF-Tucker-ALS & slow & low & low & high & no \\
				
				CMTF-OPT & slow & low & low & high & no \\
				
				\midrule
				
				\textbf{$\mathbf{S^3}$CMTF-naive} & fast & \textbf{high} & \textbf{high} & \textbf{lower} & \textbf{yes} \\
				
				\textbf{$\mathbf{S^3}$CMTF-opt} & \textbf{faster} & \textbf{high} & \textbf{high} & low & \textbf{yes} \\
				
				\bottomrule
				
			\end{tabular}
		}
	\end{center}
	\label{tab:salesman}
\end{table}
\begin{figure*} [!t]
	\includegraphics[width=1 \textwidth]{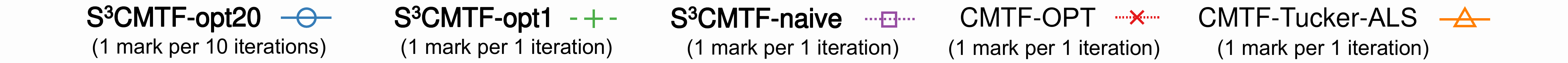}
	\subfloat[\textbf{MovieLens}]
	{	\includegraphics[width=0.30 \textwidth]{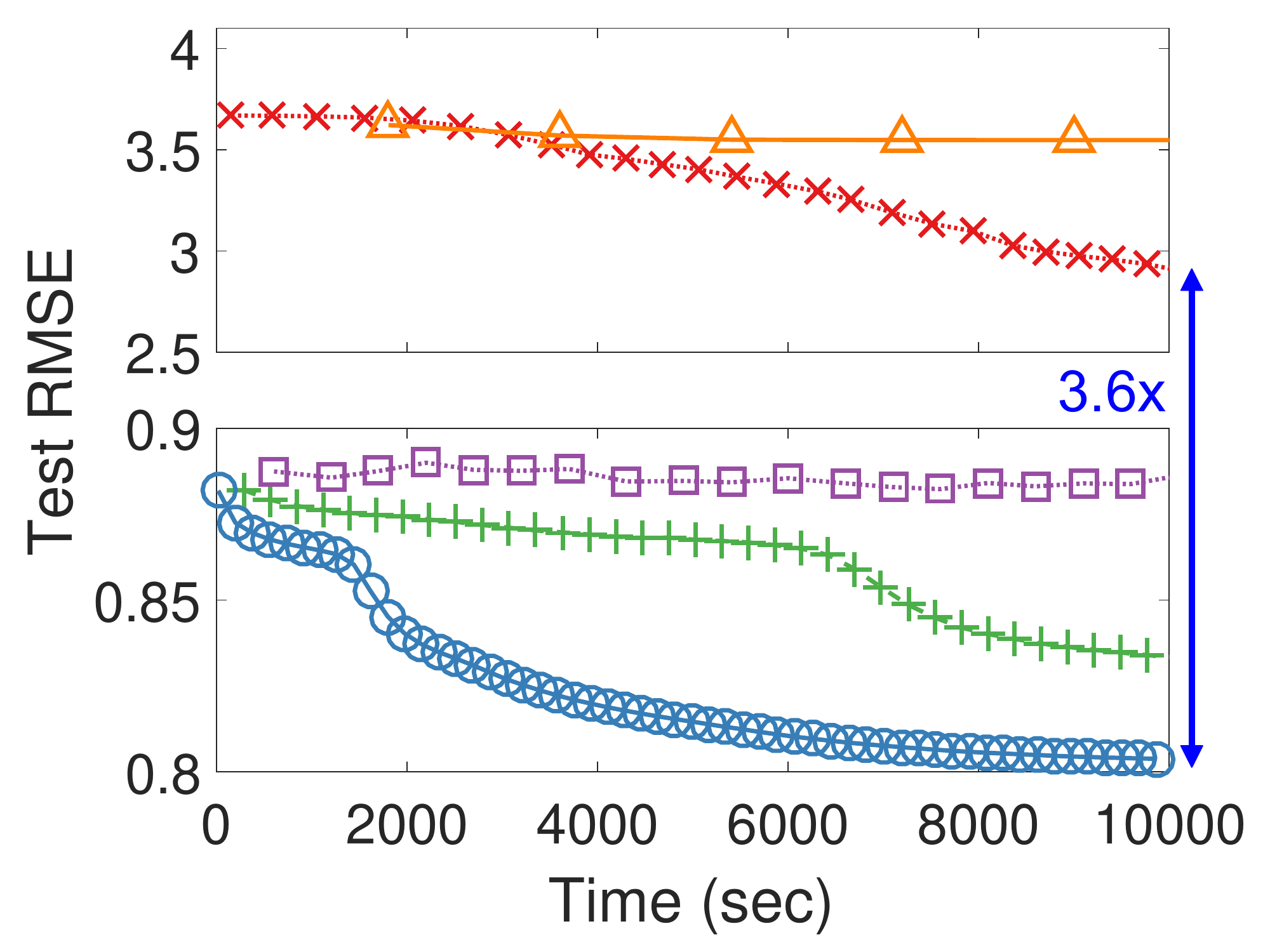}
	}
	\subfloat[\textbf{Netflix}]
	{	\includegraphics[width=0.30 \textwidth]{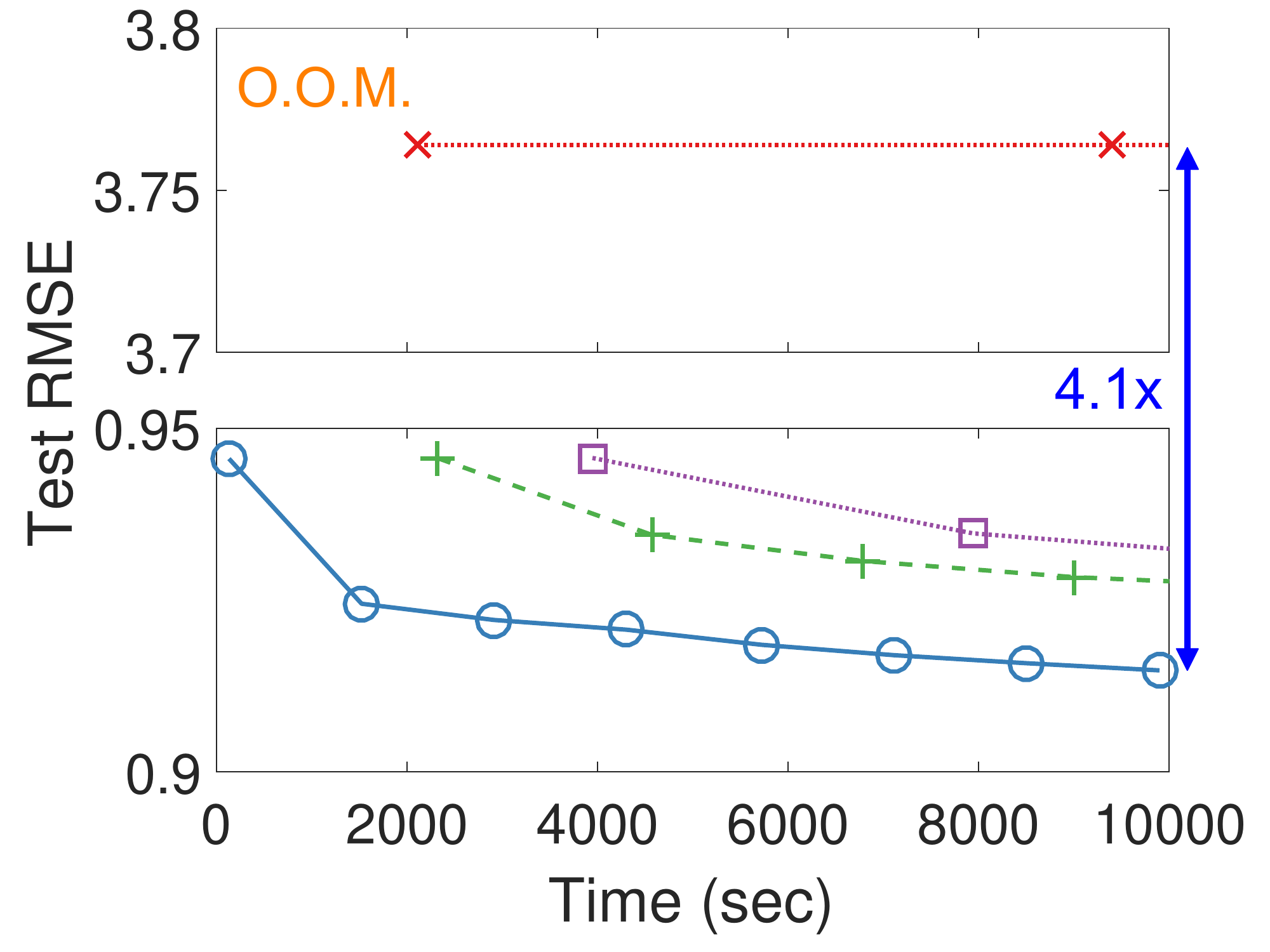}
	}
	\subfloat[\textbf{Yelp}]
	{	\includegraphics[width=0.30 \textwidth]{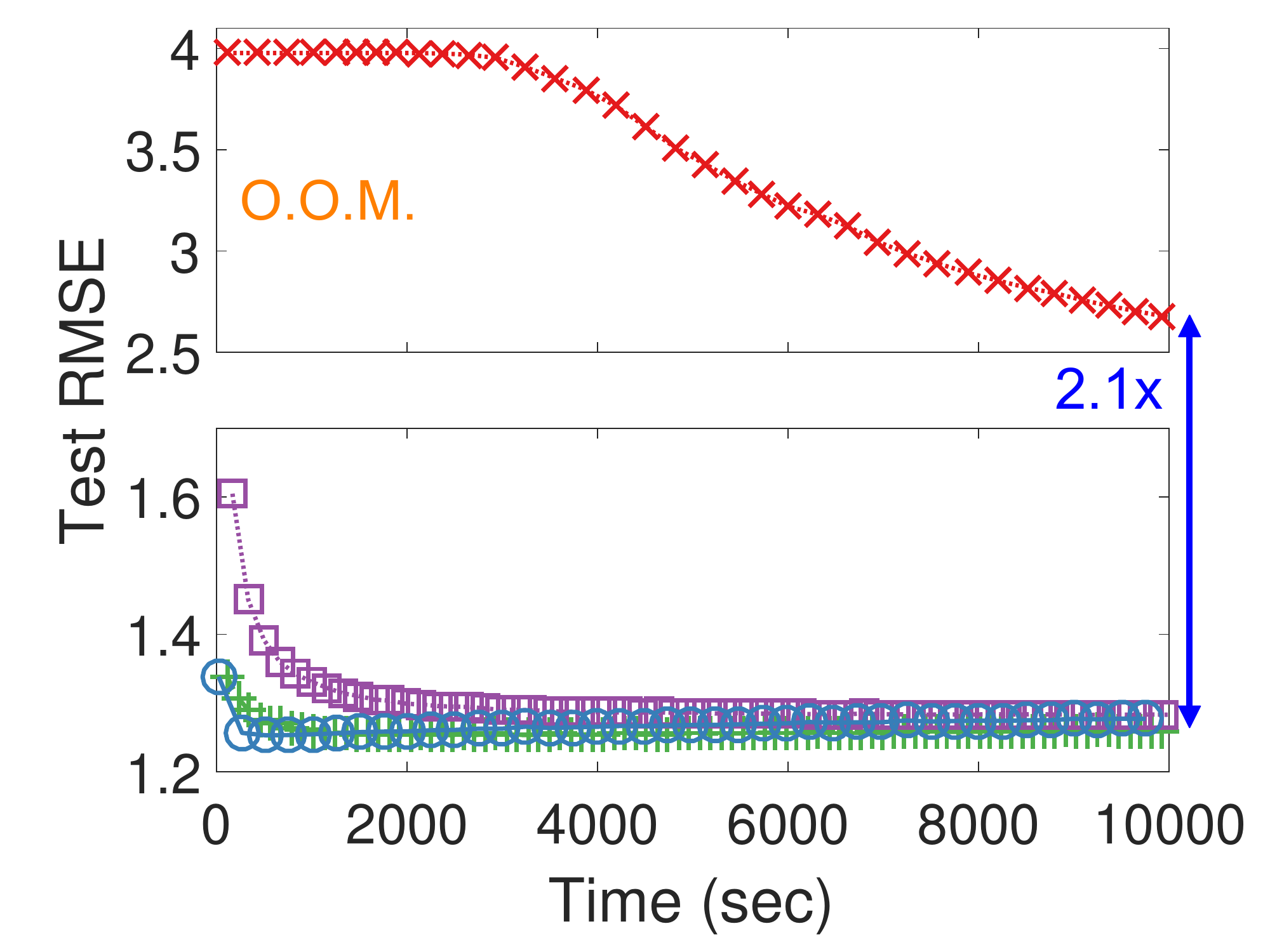}
	}
	\caption{Test RMSE of $\mathbf{S^3}$CMTF and other CMTF methods over iterations. $\mathbf{S^3}$CMTF-opt20 shows the best convergence rate and accuracy. $\mathbf{S^3}$CMTF factorizes real-world data with 2.1$\sim$4.1$\times$ less error than competitors. Note that we set one mark per 10 iterations for $\mathbf{S^3}$CMTF-opt20. O.O.M.: out of memory error.}
	\label{fig:Time Error}
\end{figure*}
\bit
\vspace{-0.1cm}
\item \textbf{Algorithm:} We propose \method, a fast, accurate, and scalable coupled tensor-matrix factorization algorithm for matrix-tensor joint datasets. \method is designed to efficiently extract factors from the joint datasets by taking advantage of sparsity, exploiting intermediate data, and parallelization.
\item \textbf{Performance:} \method empirically shows the best performance on accuracy, speed, and scalability. Especially for real-world data, \method gives \textbf{2.1$\sim$4.1$\times$ less error}, and works \textbf{11$\sim$43$\times$ faster} than existing methods as shown in Figures \ref{fig:Time Error} and \ref{fig:IterTime}.
\item \textbf{Discovery:} Applying \method on Yelp review dataset with a 3-mode tensor (user, business, time) coupled with 3 additional matrices ((user, user), (business, category), and (business, city)),
we observe interesting patterns and clusters of businesses and suggest a process for personal recommendation.
\eit

The rest of paper is organized as follows. Section \ref{sec:prelim} gives the preliminaries and related works of the tensor and CMTF. Section \ref{sec:proposed} describes our proposed \method method for fast, accurate and scalable CMTF. Section \ref{sec:experiments} shows the results of performance experiments for our proposed method. After presenting the discovery results in Section \ref{sec:case}, we conclude in Section \ref{sec:conclusions}. 
\vspace*{-0.3cm}
\section{Preliminaries and Related Works}
    \label{sec:prelim}
    In this section, we describe preliminaries for tensor and coupled matrix-tensor factorization. 
We list all symbols used in this paper in Table \ref{tab:symbol_table}.

\begin{table}[h]
	\small
	\setlength{\tabcolsep}{1pt}
	\centering
	\caption{Table of symbols.}
	\label{tab:table1}
	\begin{center}
		\begin{tabular}{cl}
			\toprule
			\textbf{Symbol} & \textbf{Definition}\\
			\midrule
			$\T{X}$ & input tensor\\
			$\T{G}$ & core tensor\\
			$N$ & order (number of modes) of the input tensor\\
			$I_{n}$ & dimensionality of $n$-th mode of input tensor $\T{X}$\\
			$J_{n}$ & dimensionality of $n$-th mode of core tensor $\T{G}$\\
			$\alpha$ & a tensor index ($i_1i_2 \cdots i_N$)\\
			$x_{\alpha}$ & the entry of $\T{X}$ with index $\alpha$\\
			$\mathbf{X}_{(n)}$ & mode-$n$ matricization of a tensor\\
			$\mathbf{U}^{(n)}$ & $n$-th factor matrix of $\T{X}$\\
			$\{\mathbf{U}\}$ & set of all factor matrices of $\T{X}$\\
			$\mathbf{u}^{(n)}_{i}$ & the $i$-th row vector of $\mathbf{U}^{(n)}$\\
			$\{\mathbf{u}\}_{\alpha}$ & ordered set of row vectors $\{\mathbf{u}^{(1)}_{i_1},\mathbf{u}^{(2)}_{i_2},\dots,\mathbf{u}^{(N)}_{i_N} \}$\\
			$\{\mathbf{u}\}_{\alpha}^{\mathsf{T}}$ & ordered set of column vectors $\{\mathbf{u}^{(1)\mathsf{T}}_{i_1},\mathbf{u}^{(2)\mathsf{T}}_{i_2},\dots,\mathbf{u}^{(N)\mathsf{T}}_{i_N} \}$\\
			$u_{ij}^{(n)}$ & entry of $\mathbf{U}^{(n)}$ with index ($i,j$)\\
			$\mathbf{Y}$ & coupled matrix\\
			$\beta$ & a matrix index $k_1k_2$\\
			$y_{\beta}$ & the entry of $\mathbf{Y}$ with index $\beta$\\
			$\mathbf{V}$ & factor matrix for the coupled matrix $\mathbf{Y}$\\
			$\mathbf{v}_{k}$ & the $k$-th row vector of $\mathbf{V}$\\
			$\Omega_{\T{X}}$ & index set of $\T{X}$\\
			$\Omega_{\T{X}}^{n,i}$ & subset of $\Omega_{\T{X}}$ having $i$ as the $n$-th index\\
			\bottomrule
		\end{tabular}
	\end{center}
	
	\label{tab:symbol_table}
	
\end{table}

\subsection{Tensor}
A tensor is a multi-dimensional array. Each \lq dimension\rq\:of a tensor is called $mode$ or $way$. The length of each mode is called \lq dimensionality\rq\:and denoted by $I_1,\cdots,I_N $.
In this paper, an $N$-mode of $N$-way tensor is denoted by the boldface Euler script capital (e.g. $\T{X}\in\mathbb{R}^{I_{1}{\times}I_{2}{\times}\cdots \times{I_{N}}}$), and matrices are denoted by boldface capitals (e.g. $\mathbf{A}$). $x_{\alpha}$ and $a_{\beta}$ denote the entry of $\T{X}$ and $\mathbf{A}$ with indices $\alpha$ and $\beta$, respectively.

We describe tensor operations used in this paper.
A mode-$n$ fiber is a vector which has fixed indices except for the $n$-th index in a tensor.
The mode-$n$ matrix product of a tensor $\T{X}\in\mathbb{R}^{I_{1}{\times}I_{2}{\times}\cdots \times{I_{N}}}$ with a matrix $\mathbf{A}\in\mathbb{R}^{J{\times}I_{n}}$ is denoted by $\T{X}{\times}_n\mathbf{A}$ and has the size of $I_{1}{\times}{\cdots}I_{n-1}{\times}J{\times}I_{n+1}\cdots\times{I_{N}}$. It is defined:
\begin{equation} \label{eqn:n mode matrix product elementwise}
(\T{X} \times_n \mathbf{A})_{i_1 \dots i_{n-1} j i_{n+1} \dots i_N} = \sum\limits_{i_n=1}^{I_n} x_{i_1 i_2 \dots i_N}a_{ji_n}
\end{equation}
where $a_{ji_n}$ is the $(j, i_n)$-th entry of $\mathbf{A}$.
For brevity, we use following shorthand notation for multiplication on every mode as in \cite{kolda2008scalable}:
\begin{equation} \label{eqn:n mode matrix set product}
\T{X}\times\{\mathbf{A}\}:=\T{X}\times_{1} \mathbf{A}^{(1)} \times_2 \mathbf{A}^{(2)} \cdots \times_N \mathbf{A}^{(N)}
\end{equation}
where $\{\mathbf{A}\}$ denotes the ordered set $\{\mathbf{A^{(1)}},\mathbf{A}^{(2)},\cdots,,\mathbf{A}^{(N)}\}$.

We use the following notation for multiplication on every mode except $n$-th mode.
\begin{equation*}
\T{X}\times_{-n}\{\mathbf{A}\}:=\T{X}\times_{1} \mathbf{A}^{(1)} \cdots \times_{n-1} \mathbf{A}^{(n-1)} \times_{n+1} \mathbf{A}^{(n+1)} \cdots \times_N \mathbf{A}^{(N)}
\end{equation*}
We examine the case that an ordered set of row vectors $\{\mathbf{a
	^{(1)}},\mathbf{a^{(2)}},\cdots$ $,\mathbf{a^{(N)}}\}$, denoted by $\{\mathbf{a}\}$, is multiplied to a tensor $\T{X}$.
First, consider the multiplication for every corresponding mode. By Equation \eqref{eqn:n mode matrix product elementwise},
\begin{equation*}
\T{X}\times\{\mathbf{a}\}=\sum_{i_1=1}^{I_1}\sum_{i_2=1}^{I_2}\cdots\sum_{i_N=1}^{I_N}x_{i_1i_2\cdots i_N}a_{i_1}^{(1)}a_{i_2}^{(2)}\cdots a_{i_N}^{(N)}
\end{equation*}
where $a_{k}^{(m)}$ denotes the $k$-th element of $\mathbf{a}^{(m)}$. Then, consider the multiplication for every mode except $n$-th mode. Such multiplication results to a vector of length $I_n$. The $k$-th entry of the vector is
\begin{equation}\label{eqn:n mode except row set product}
\big[\T{X}\times_{-n}\{\mathbf{a}\}\big]_{k}=\sum_{\forall\alpha\in\Omega_{\T{X}}^{n,k}}{x_{\alpha}a^{(1)}_{i_1}\cdots a^{(n-1)}_{i_{n-1}}a^{(n+1)}_{i_{n+1}}\cdots a^{(N)}_{i_{N}}}
\end{equation}
where $\Omega_{\T{X}}^{n,k}$ denotes the index set of $\T{X}$ having its $n$-th index as $k$. $\alpha=(i_1i_2\cdots i_N)$ denotes the index for an entry.

\subsection{Tucker Decomposition}
Tucker decomposition is one of the most popular tensor factorization models and is also known as Tucker decomposition. Tucker decomposition approximates an $N$-mode tensor $\T{X}\in\mathbb{R}^{I_{1}{\times}I_{2}{\times}\cdots \times{I_{N}}}$ into a core tensor $\T{G}\in\mathbb{R}^{J_{1}{\times}J_{2}{\times}\cdots \times{J_{N}}}$ and factor matrices $\mathbf{U}^{(1)}\in\mathbb{R}^{I_{1}{\times}J_{1}},\mathbf{U}^{(2)}$\\$\in\mathbb{R}^{I_{2}{\times}J_{2}},\dots,\mathbf{U}^{(N)}\in\mathbb{R}^{I_{N}{\times}J_{N}}$ satisfying
\begin{equation*}
\T{X}\approx\tilde{\T{X}}=
\T{G}\times_{1} \mathbf{U}^{(1)} \times_2 \mathbf{U}^{(2)} \dots \times_N \mathbf{U}^{(N)}
=\T{G}\times\{\mathbf{U}\}
\end{equation*}
Element-wise formulation of Tucker model is
\begin{equation} \label{eqn:hosvd elementwise}
\begin{split}
x_{\alpha}\approx\tilde{x}_{\alpha}
&=\sum_{j_1=1}^{J_1}\sum_{j_2=1}^{J_2}\cdots\sum_{j_N=1}^{J_N}g_{j_1j_2\cdots j_N}u_{i_1j_1}^{(1)}u_{i_2j_2}^{(2)}\cdots u_{i_Nj_N}^{(N)}\\
&=\T{G}\times_{1} \mathbf{u}_{i_1}^{(1)} \times_2 \mathbf{u}_{i_2}^{(2)} \dots \times_N \mathbf{u}_{i_N}^{(N)}
:=\T{G}\times\{\mathbf{u}\}_\alpha
\end{split}
\end{equation}
where $\alpha$ is a tensor index $(i_1i_2\cdots i_N)$, and $\mathbf{u}_{i_n}^{(n)}$ denotes the $i_n$-th row of factor matrix $\mathbf{U}^{(n)}$. $\{\mathbf{u}\}_\alpha$ denotes the set of factor rows $\{\mathbf{u}_{i_1}^{(1)},\mathbf{u}_{i_2}^{(2)},\cdots,\mathbf{u}_{i_N}^{(N)}\}$.
Note that the core tensor $\T{G}$ implies the relation between the factors in Tucker formulation.
When the core tensor size satisfies $J_1=J_2=\cdots=J_N$ and the core tensor $\T{G}$ is hyper-diagonal, it is equivalent to CANDECOMP/PARAFAC (CP) decomposition.
There is orthogonality constraint for Tucker decomposition: each factor matrix is a column-wise orthogonal matrix (e.g. $\mathbf{U}^{(n)T} \mathbf{U}^{(n)} = \mathbf{I}$ for $n=1,\cdots,N$ where $\mathbf{I}$ is an identity matrix).

\subsection{Coupled Matrix-Tensor Factorization}\label{subsec:CMTF}
Coupled matrix-tensor factorization (CMTF) is proposed for collective factorization of a tensor and matrices. CMTF integrates matrix factorization and tensor factorization.
\begin{definition}
	\setlength{\abovedisplayskip}{0.03cm}
	\setlength{\belowdisplayskip}{0.03cm}
	\textbf{(Coupled Matrix-Tensor Factorization)}
	Given an N-mode tensor $\T{X}\in\mathbb{R}^{I_1\times\cdots\times I_N}$ and a matrix $\mathbf{Y}\in\mathbb{R}^{I_c\times K}$ where $c$ is the coupled mode, $\T{X}\approx\tilde{\T{X}}=\T{G}\times\{\mathbf{U}\}$, $\mathbf{Y}\approx\tilde{\mathbf{Y}}=\mathbf{U}^{(c)}\mathbf{V}^{\mathsf{T}}$ are the coupled matrix-tensor factorization.
	$\mathbf{U}^{(c)}\in\mathbb{R}^{I_c\times J_c}$ is the $c$-th mode factor matrix, and $\mathbf{V}\in\mathbb{R}^{K\times J_c}$ denotes the factor matrix for coupled matrix. Finding the factor matrices and core tensor for CMTF is equivalent to solving
	\begin{equation} \label{eqn:cmtf equation}
	\underset{\mathbf{U}^{(1)},\cdots,\mathbf{U}^{(N)},\mathbf{V},\T{G}} {\arg \min}\|\T{X}-\T{G}\times\{\mathbf{U}\}\|^2+\|\mathbf{Y}-\mathbf{U}^{(c)}\mathbf{V}^{\mathsf{T}}\|^2
	\end{equation}
	where $\|\bullet\|$ denotes the Frobenius norm.
\end{definition}

Various methods have been proposed to efficiently solve the CMTF problem.
An alternating least squares (ALS) method CMTF-Tucker-ALS \cite{ozcaglar2012algorithmic} is proposed.
CMTF-Tucker-ALS is based on Tucker-ALS (HOOI) \cite{de2000best} which is a popular method for solving Tucker model. Tucker-ALS suffers from a crucial intermediate memory-bottleneck problem known as \textit{M-bottleneck problem} \cite{oh2017s} that arises from materialization of a large dense tensor $\T{X}\times_{-n}\{\mathbf{U}\}^{\mathsf{T}}$ as intermediate data where $\{\mathbf{U}\}^{\mathsf{T}}=\{\mathbf{U}^{(1)\mathsf{T}},\mathbf{U}^{(2)\mathsf{T}},\cdots,\mathbf{U}^{(N)\mathsf{T}}\}$.

Most existing methods use CP decomposition model for $\tilde{\T{X}}$ where $J_1=J_2=\cdots=J_N$ and the core tensor $\T{G}$ is hyper-diagonal \cite{acar2011all, jeon2016scout, jeon2015haten2, papalexakis2014turbo, beutel2014flexifact}.
CMTF-OPT \cite{acar2011all} is a representative algorithm for CMTF using CP decomposition model which uses gradient descent method to find factors.
HaTen2 \cite{jeon2015haten2,DBLP:journals/vldb/JeonPFSK16}, and SCouT \cite{jeon2016scout} propose distributed methods for CMTF using CP decomposition model.
Turbo-SMT \cite{papalexakis2014turbo} provides a time-boosting technique for CP-based CMTF methods.

Note that Equation \eqref{eqn:cmtf equation} requires entire data entries of $\T{X}$ and $\mathbf{Y}$.
It shows low accuracy when $\T{X}$ and $\mathbf{Y}$ are sparse since empty entries are set to zeros even when they are irrelevant. For example, an empty entry in movie rating data does not mean score 0. For the reason above methods show low accuracy for real-world sparse data; what we focus on this paper is solving CMTF for sparse data.
\begin{definition}\label{def:sparse cmtf}
	\setlength{\abovedisplayskip}{0.03cm}
	\setlength{\belowdisplayskip}{0.03cm}
	\textbf{(Sparse CMTF)}
	When $\T{X}$ and $\mathbf{Y}$ are sparse, sparse CMTF aims to find factors only considering observed entries. Let $\T{W}^{(1)}$ and $\mathbf{W}^{(2)}$ indicate the observed entries of $\T{X}$ and $\mathbf{Y}$ such that
	\begin{alignat*}{2}
	w^{(1)}_{\alpha} \big(w^{(2)}_{\beta}\big)=
	\begin{cases}
	1 &\text{if $x_{\alpha}$\big($y_{\beta}$\big) is known}\\
	0 &\text{if $x_{\alpha}$\big($y_{\beta}$\big) is missing}
	\end{cases} && \text{, for } \forall \alpha\in\Omega_{\T{X}}\big(\forall\beta\in\Omega_{\mathbf{Y}}\big)
	\end{alignat*}
	We modify Equation \eqref{eqn:cmtf equation} as
	\begin{equation} \label{eqn:sparse cmtf equation}
	\underset{\mathbf{U}^{(1)},\cdots,\mathbf{U}^{(N)},\mathbf{V},\T{G}} {\arg \min}\|\T{W}^{(1)}\ast(\T{X}-\T{G}\times\{\mathbf{U}\})\|^2+\|\mathbf{W}^{(2)}\ast(\mathbf{Y}-\mathbf{U}^{(c)}\mathbf{V}^{\mathsf{T}})\|^2
	\end{equation}
	where $\ast$ denotes the Hadamard product (element-wise product).
\end{definition}

CMTF-Tucker-ALS does not support sparse CMTF. For CP model, CMTF-OPT provides single machine approach for sparse CMTF, and CDTF \cite{DBLP:journals/tkde/ShinSK17} and FlexiFaCT \cite{beutel2014flexifact} provide distributed methods for sparse CMTF. However, CP model suffers from high error because it does not capture the correlations between different factors of different modes because its core tensor has only hyper-diagonal nonzero entries \cite{kiers1997uniqueness}.

\section{Proposed Method}
    \label{sec:proposed}
    \begin{figure} [t]
	\begin{center}
		\includegraphics[width=0.48 \textwidth]{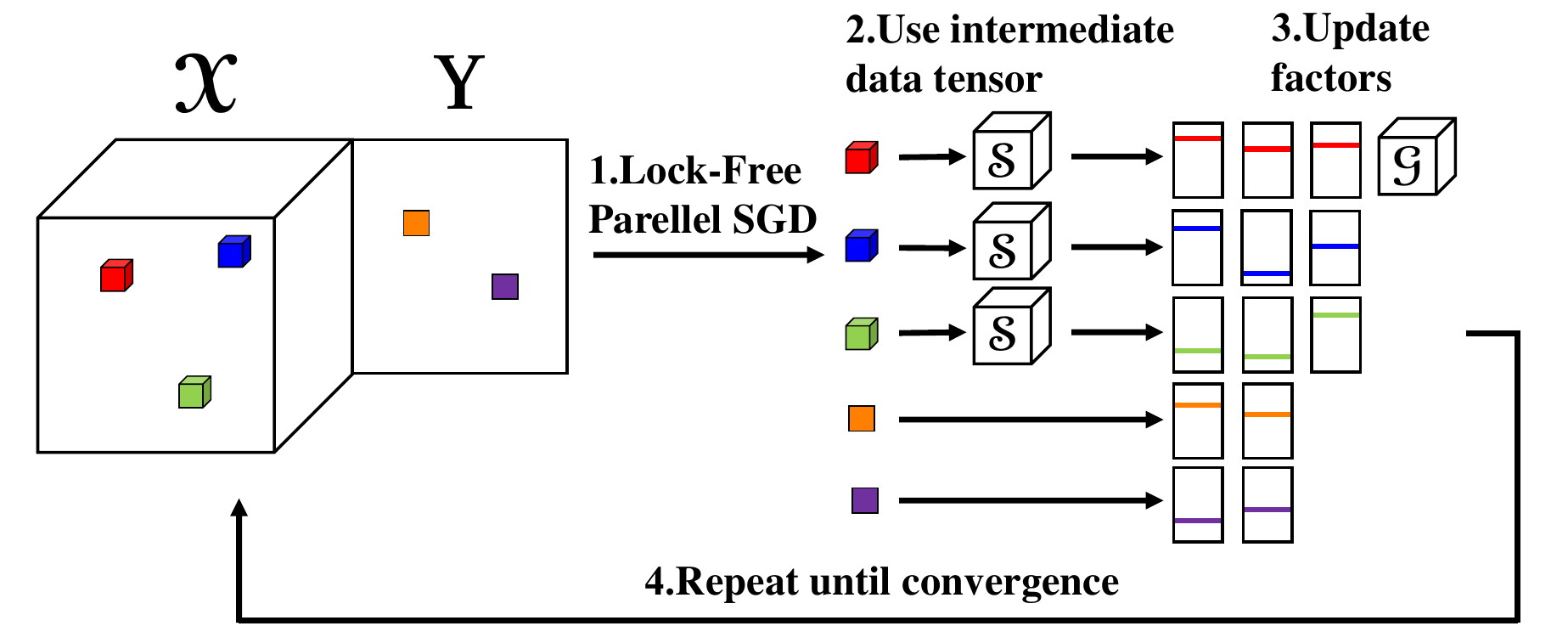}
	\end{center}
	\caption{The scheme for $\mathbf{S^3}$CMTF.}
	\label{fig:scheme}
\end{figure}
\subsection{Overview}
In this section, we describe \method (Sparse, lock-free SGD based, and Scalable CMTF), our proposed method for fast, accurate, and scalable CMTF method.
CMTF methods for dense data are prone to get high errors because of zero-filling for empty entries. On the other hand, CP-based methods show high prediction error because of simplicity of the model \cite{kiers1997uniqueness}.
Our purpose is to devise an improved sparse CMTF model and propose a fast and scalable algorithm for the model.

We propose a basic version of our method \method-naive and a time-improved version \method-opt. Figure \ref{fig:scheme} shows the overall scheme for \method. \method-naive adopts lock-free parallel SGD for the parallel update, and \method-opt further improves the speed of \method-naive by exploiting intermediate data and reusing them.

\subsection{Objective Function \& Gradient}\label{subsec:objective function}
We discuss the improved formulation of the sparse CMTF problem defined in Definition \ref{def:sparse cmtf}. For simplicity, we consider the case that one matrix $\mathbf{Y}\in\mathbb{R}^{I_c\times K}$ is coupled to the $c$-th mode of a tensor $\T{X}\in\mathbb{R}^{I_1\times\cdots\times I_N}$.
Equation \eqref{eqn:sparse cmtf equation} takes excessive time and memory because it includes materialization of dense tensor $\T{G\times{\{\mathbf{U}\}}}$. Therefore, we formulate the new CMTF objective function $f$ to exploit the sparsity of data. $f$ is the weighted sum of two functions $f_t$ and $f_m$ where they are element-wise sums of squared reconstruction error and regularization terms of tensor $\T{X}$ and matrix $\mathbf{Y}$, respectively.
\begin{equation}\label{eqn:cost function}
f=\frac{1}{2}f_t+\frac{\lambda_m}{2}f_m
\end{equation}
where $\lambda_{m}$ is a balancing factor of two functions.
\begin{equation*}
f_t=\Big[\sum\limits_{\forall\alpha\in\Omega_{\T{X}}}\big(x_{\alpha}-(\T{G}\times\{\mathbf{u}\}_{\alpha})\big)^2\Big]+\lambda_{reg}\Big(\|\T{G}\|^2+\sum_{n=1}^{N}\|\mathbf{U}^{(n)}\|^2\Big)
\end{equation*}
where $\alpha=(i_{1}\cdots i_{N})$, $\Omega_{\T{X}}$ is the nonzero index set of $\T{X}$, and $\lambda_{reg}$ denotes the regularization parameter for factors. We rewrite the equation so that it is amenable to SGD update.
\begin{equation*}
f_{t}= \sum\limits_{\forall\alpha\in\Omega_{\T{X}}}\Big[\big(x_{\alpha}-(\T{G}\times\{\mathbf{u}\}_{\alpha})\big)^2+\frac{\lambda_{reg}}{|\Omega_{\T{X}}|}\|\T{G}\|^2+\lambda_{reg}\sum_{n=1}^{N}\frac{\|\mathbf{u}_{i_n}^{(n)}\|^2}{|\Omega_{\T{X}}^{n,i_n}|}\Big]
\end{equation*}
where $\alpha=(i_{1}\cdots i_{N})$. Note that $\Omega_{\T{X}}^{n,i_n}$ is the subset of $\Omega_{\T{X}}$ having $i_n$ as the $n$-th index. Now we formulate $f_{m}$, the sum of squared errors of coupled matrix and regularization term corresponding to the coupled matrix.
\begin{equation*}
f_{m}=\sum_{\forall\beta=(j_{1}j_{2})\in\Omega_{\mathbf{Y}}}\Big[\big(y_{\beta}-\mathbf{u}^{(c)}_{j_1}\mathbf{v}_{j_2}^{\mathsf{T}}\big)^2+\frac{\lambda_{reg}}{|\Omega_{\mathbf{Y}}^{2,j_2}|}\|\mathbf{v}_{j_2}\|^2\Big]
\end{equation*}
We calculate the gradient of $f$ (Equation \eqref{eqn:cost function}) with respect to factors for stochastic gradient descent update. Consider that we pick one index among tensor index $\alpha=(i_1\cdots i_N)\in\Omega_{\T{X}}$ and matrix index $\beta=(j_{1}j_{2})\in\Omega_{\mathbf{Y}}$. We calculate the corresponding partial derivatives of $f$ with respect to the factors and the core tensor as follows.
\begin{equation}\label{eqn:gradient}
\begin{split}
\left.\frac{\partial f}{\partial \mathbf{u}^{(n)}_{i_n}}\right\vert_{\alpha}&=
-\big(x_{\alpha}-(\T{G}\times\{\mathbf{u}\}_{\alpha})\big)\big{\lbrack}(\T{G}\times_{-n}\{\mathbf{u}\}_{\alpha})_{(n)}\big{\rbrack}^{\mathsf{T}}+\frac{\lambda_{reg}}{|\Omega_{\T{X}}^{n,i_n}|}\mathbf{u}^{(n)}_{i_n}\\
\left.\frac{\partial f}{\partial \T{G}}\right\vert_{\alpha}&=-\big(x_{\alpha}-(\T{G}\times\{\mathbf{u}\}_{\alpha})\big)\times\{\mathbf{u}\}^{\mathsf{T}}_{\alpha}+\frac{\lambda_{reg}}{|\Omega_{\T{X}}|}\T{G}\\
\left.\frac{\partial f}{\partial \mathbf{u}^{(c)}_{j_1}}\right\vert_{\beta}&=-\lambda_m(y_{\beta}-\mathbf{u}^{(c)}_{j_1}\mathbf{v}_{j_2}^{\mathsf{T}})\mathbf{v}_{j_2}\\
\left.\frac{\partial f}{\partial \mathbf{v}_{j_2}}\right\vert_{\beta}&=-\lambda_{m}(y_{\beta}-\mathbf{u}^{(c)}_{j_1}\mathbf{v}_{j_2}^{\mathsf{T}})\mathbf{u}^{(c)}_{j_1}+\frac{\lambda_{m}\lambda_{reg}}{|\Omega_{\mathbf{Y}}^{2,j_2}|}\mathbf{v}_{j_2}
\end{split}
\end{equation}

We omit the detailed derivation of Equations \eqref{eqn:gradient} for brevity. Note that our formulated coupled matrix-tensor factorization model is also applicable to dense data and easily generalized to the case that multiple matrices are coupled to a tensor.
We couple multiple matrices to a tensor for experiments in Sections \ref{sec:experiments} and \ref{sec:case}.

\subsection{Lock-Free Parallel Update}\label{subsec:parallel update}
How can we parallelize the SGD updates in multiple cores? In general, SGD approach is hard to be parallelized because each parallel update may suffer from memory conflicts by attempting to write the same variables to memory concurrently \cite{bradley2011parallel}.
One solution for this problem is memory locking and synchronization. However, there are much overhead associated with locking. Therefore, we use lock-free strategy to parallelize \method.
We develop parallel update scheme for \method by adapting \hogwild update scheme \cite{recht2011hogwild}.
\begin{figure} [t]
	\begin{center}
		\includegraphics[width=0.48 \textwidth]{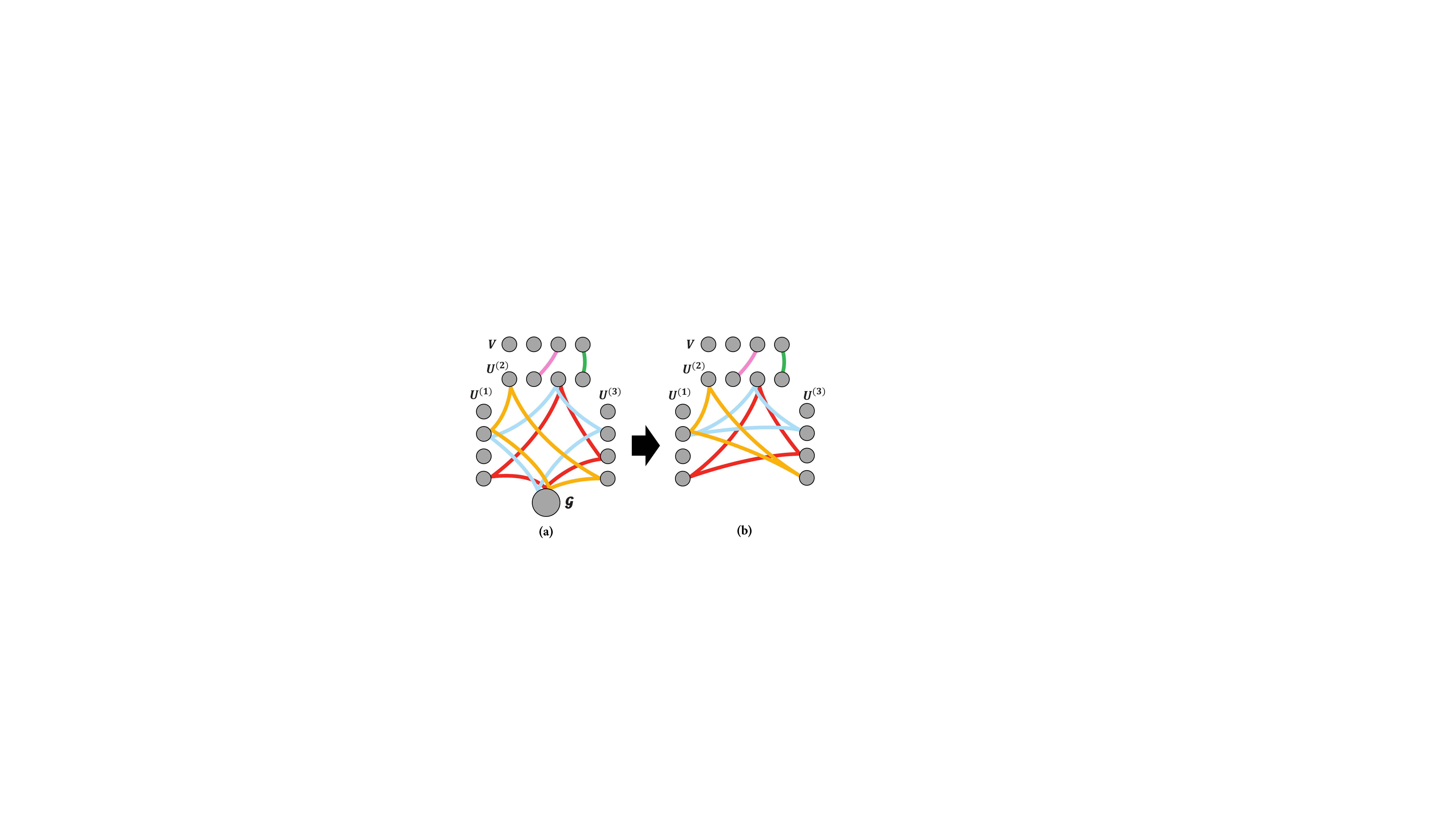}
	\end{center}
	\caption{Example graphs induced by $\mathbf{S^3}$CMTF objective function (Equation \eqref{eqn:cost function}). A matrix $\mathbf{Y}$ is coupled to the second mode of $\T{X}$ with a factor matrix $\mathbf{V}$. Each node represents a factor row or the core tensor. Each hyperedge includes corresponding factors to an SGD update. (a) Induced hypergraph with core tensor. Every hyperedge corresponding to tensor entries includes $\T{G}$. (b) Induced hypergraph without core tensor. The graph reveals sparsity as every node is shared by only few hyperedges.}
	\label{fig:inducedGraph}
\end{figure}

\begin{definition}
	\textbf{(Induced Hypergraph)}
	The objective function in Equation \eqref{eqn:cost function} induces a hypergraph $G=(V,E)$ whose nodes represent factor rows and core tensor. Each entry of $\T{X}$ and $\mathbf{Y}$ induces a hyperedge $e\in E$ consisting of corresponding factor rows or core tensor. Figure \ref{fig:inducedGraph}a shows an example induced graph of \method.
\end{definition}

Lock-free parallel update guarantees near linear convergence property of a sparse SGD problem in which conflicts between different updates rarely occur \cite{recht2011hogwild}.
However, in our formulation, every update of tensor entries includes the core tensor $\T{G}$ as shown in Figure \ref{fig:inducedGraph}a.
We allocate the update of core tensor $\T{G}$ to one core to solve the problem. Then we obtain a new induced hypergraph in Figure \ref{fig:inducedGraph}b.
The newly obtained hypergraph satisfies the sparsity condition for convergence. Lemma \ref{lemma:convergence} proves the convergence property of parallel updates.

\begin{lemma}\label{lemma:convergence}
	\textbf{(Convergence)}
	If we assume that the elements of the tensor $\T{X}$ and coupled matrix $\mathbf{Y}$ are sampled uniformly at random, lock-free parallel update of \method converges to a local optimum.
\end{lemma}

\begin{proof}
	For brevity, we assume that the dimension and rank of each mode are $I$ and $J$, respectively. We use the notations used in Equation (2.6) of \cite{recht2011hogwild}. For a given hypergraph $G=(V,E)$, we define
	\begin{equation*}
	\Omega := \max_{e\in E}|e|,
	\Delta := \frac{ \max_{v\in V}|\{e \in E : v \in e\}|}{|E|}
	\end{equation*}
	\begin{equation*}
	\rho := \frac{ \max_{e\in E}|\{\tilde{e} \in E : \tilde{e} \cap e\neq\emptyset\}|}{|E|}
	\end{equation*}
	First, consider the case when the tensor order is 2. $\Omega$ has the same value, and $\Delta$ has doubled value of the matrix factorization problem in \cite{recht2011hogwild}: $\Omega\approx2J$, $\Delta\approx\frac{2\log(I)}{I}$. $\rho$ naturally satisfies $\rho\approx\frac{3\log{(I)}}{I}$. Next, when the tensor order is N, $\Omega$ linearly scales up and $\Omega\approx NJ$, $\Delta$ and $\rho$ stay same: $\Delta \approx \frac{2\log(I)}{I}$, $\rho \approx \frac{3\log(I)}{I}$. Parallel update converges as proved in Proposition 4.1 of \cite{recht2011hogwild}.
\end{proof}

\begin{algorithm} [t]
	\small
	\caption{\method-naive} \label{alg:s3cmtf}
	\begin{algorithmic}[1]
		\REQUIRE Tensor $\T{X} \in \mathbb{R}^{I_1 \times \cdots \times I_N}$, rank $(J_1,\cdots ,J_N)$, number of parallel cores $P$, initial learning rate $\eta_0$, decay rate $\mu$, coupled mode $c$, and coupled matrix $\mathbf{Y} \in \mathbb{R}^{I_c \times K}$
		
		\ENSURE Core tensor $\T{G} \in \mathbb{R}^{J_1 \times \cdots \times J_N}$, factor matrices $\mathbf{U}^{(1)},\cdots,\mathbf{U}^{(N)}$, $\mathbf{V}$
		
		\STATE Initialize $\T{G}$, $\mathbf{U}^{(n)} \in \mathbb{R}^{I_n \times J_n}$ for $n= 1,\cdots,N$, and $\mathbf{V}$ randomly
		
		\REPEAT
		
		\FORP{$\forall\alpha=(i_1\cdots i_N)\in\Omega_{\T{X}}$, $\forall\beta=(j_1j_2)\in\Omega_{\mathbf{Y}}$ in random order}
		
		\IF{$\alpha$ is picked}
		
		\STATE ($\frac{\partial{f}}{\partial{\mathbf{u}_{i_1}^{(1)}}}$,$\cdots$,$\frac{\partial{f}}{\partial{\mathbf{u}_{i_N}^{(N)}}}$,$\frac{\partial{f}}{\partial{\T{G}}}$)
		$\leftarrow$\textit{compute\_gradient}($\alpha$,$x_\alpha$,$\T{G}$)
		
		
		\STATE $\mathbf{u}^{(n)}_{i_n} \leftarrow \mathbf{u}^{(n)}_{i_n}-\eta_t\frac{\partial{f}}{\partial{\mathbf{u}_{i_n}^{(n)}}}$, (for $n=1, \cdots, N$)
		
		
		\STATE $\T{G} \leftarrow \T{G}-\eta_t P \frac{\partial{f}}{\partial{\T{G}}}$
		(executed by only one core)
		
		\ENDIF
		
		\IF{$\beta$ is picked}
		
		\STATE
		$\tilde{y}_{\beta}\leftarrow \mathbf{u}^{c}_{j_1}\mathbf{v}^{\mathsf{T}}_{j_2}$,~
		$\frac{\partial{f}}{\partial{\mathbf{u}_{j_1}^{(c)}}}
		\leftarrow
		-\lambda_m(y_{\beta}-\tilde{y}_{\beta})\mathbf{v}_{j_2}$
		
		\STATE
		$\frac{\partial{f}}{\partial{\mathbf{v}_{j_2}}}
		\leftarrow
		-\lambda_{m}(y_{\beta}-\tilde{y}_{\beta})\mathbf{u}^{(c)}_{j_1}+\frac{\lambda_{m}\lambda_{reg}}{|\Omega_{\mathbf{Y}_{2,j_2}}|}\mathbf{v}_{j_2}$
		
		\STATE $\mathbf{u}^{(c)}_{j_1} \leftarrow \mathbf{u}^{(c)}_{j_1}-\eta_t\frac{\partial{f}}{\partial{\mathbf{u}_{j_1}^{(c)}}}$,
		$\mathbf{v}_{j_2} \leftarrow \mathbf{v}_{j_2}-\eta_t\frac{\partial{f}}{\partial{\mathbf{v}_{j_2}}}$
		
		\ENDIF
		
		\ENDFOR
		
		\STATE $\eta_t=\eta_0(1+\mu t)^{-1}$
		
		\UNTIL{convergence conditions are satisfied}
		
		\FOR{$n = 1,\dots,N$}
		
		\STATE $\mathbf{Q}^{(n)}$,$\mathbf{R}^{(n)} \leftarrow$ QR decomposition of $\mathbf{U}^{(n)}$
		
		\STATE $\mathbf{U}^{(n)} \leftarrow \mathbf{Q}^{(n)}$
		, $\T{G}\leftarrow\T{G}\times_{n}\mathbf{R}^{(n)}$
		
		\ENDFOR
		
		\STATE $\mathbf{V}\leftarrow\mathbf{V}\mathbf{R}^{(c)\mathsf{T}}$
		
		\RETURN $\T{G}$, $\mathbf{U}^{(1)},\cdots,\mathbf{U}^{(N)}, \mathbf{V}$
		
	\end{algorithmic}
\end{algorithm}
\vspace*{-0.3cm}
\subsection{$\mathbf{S^3}$CMTF-naive}
We present a basic version of our method, \method-naive. \method-naive solves the sparse CMTF problem by parallel SGD techniques explained in Sections \ref{subsec:objective function}-\ref{subsec:parallel update}.
Algorithm \ref{alg:s3cmtf} shows the procedure of \method-naive.
In the beginning, \method-naive initializes factor matrices and core tensor randomly (line 1 of Algorithm \ref{alg:s3cmtf}). The outer loop (lines 2-16) repeats until the factor variables converge. The inner loop (lines 3-14) is conducted by several cores in parallel except for line 7. In each inner loop, \method-naive selects an index which belongs to $\Omega_{\T{X}}$ or $\Omega_{\mathbf{Y}}$ in random order (line 3). If a tensor index $\alpha$ is picked, then the algorithm calculates the partial gradients of corresponding factor rows using \textit{compute\_gradient} (Algorithm \ref{alg:compute_gradient}) in line 5, and updates factor row vectors (line 6).
Core tensor $\T{G}$ is updated by only one core (line 7);
the number $P$ of cores is multiplied to the gradient to compensate for the one-core update so that SGD uses the same learning rate for all the parameters.
If a coupled matrix index $\beta$ is picked, then the gradient update is conducted on corresponding factor row vectors (lines 9-13). At the end of the outer loop, the learning rate $\eta_t$ is monotonically decreased \cite{bottou2012stochastic}. (line 15). QR decomposition is applied on factors to satisfy orthogonality constraint of factor matrices (lines 17-20).
QR decomposition of $\mathbf{U}^{(n)}$ generates $\mathbf{Q}^{(n)}$, an orthogonal matrix of the same size as $\mathbf{U}^{(n)}$, and a square matrix $\mathbf{R}^{(n)}\in \mathbb{R}^{J_n \times J_n}$.
Substituting $\mathbf{U}^{(n)}$ by $\mathbf{Q}^{(n)}$ (line 19) and $\T{G}$ by $\T{G}\times_{1}\mathbf{R}^{(1)}\cdots\times_{N}\mathbf{R}^{(N)}$ (after $N$-th execution of line 19) result in an equivalent factorization~\cite{kolda2006multilinear}. 
In the same manner, we substitute $\mathbf{V}$ by $\mathbf{VR}^{(c)\mathsf{T}}$ (line 21) because $\mathbf{\tilde{Y}}=\mathbf{U}^{(c)}\mathbf{V}^{\mathsf{T}}
=\mathbf{Q}^{(c)}\mathbf{R}^{(c)}\mathbf{V}^{\mathsf{T}}
=\mathbf{Q}^{(c)}(\mathbf{V}\mathbf{R}^{(c)\mathsf{T}})^{\mathsf{T}}
$.
\vspace*{-0.2cm}

\subsection{$\mathbf{S^3}$CMTF-opt}
\textbf{Reusing the intermediate data.}
There are many redundant calculations in \method-naive. For example, $\T{G}\times_{-n}\{\mathbf{u}\}_{\alpha}$ is calculated for every execution of \textit{compute\_gradient} (Algorithm \ref{alg:compute_gradient}) in line 5 of Algorithm \ref{alg:s3cmtf}. In \method-opt, we save the time by storing the intermediate data of calculating $\tilde{x}_\alpha$ and reusing them.


\begin{definition}
	\textbf{(Intermediate Data)} When updating the factor rows for a tensor entry $x_{\alpha=(i_1\cdots i_N)}$, we define ($j_1j_2\cdots j_N$)-th element of intermediate data $\T{S}$:\\
	\centerline{$s_{j_1j_2\cdots j_N}\leftarrow g_{j_1j_2\cdots j_N}u_{i_1j_1}^{(1)}u_{i_2j_2}^{(2)}\cdots u_{i_Nj_N}^{(N)}$}
\end{definition}

\begin{algorithm} [t]
	\small
	\caption{\textit{compute\_gradient}($\alpha$,$x_\alpha$,$\T{G}$)} \label{alg:compute_gradient}
	\begin{algorithmic}[1]
		\REQUIRE Tensor entry $x_\alpha$, $\alpha=(i_1\cdots i_N)$$\in\Omega_{\T{X}}$, core tensor $\T{G}$
		
		\ENSURE Gradients $\frac{\partial{f}}{\partial{\mathbf{u}_{i_1}^{(1)}}}$,$\frac{\partial{f}}{\partial{\mathbf{u}_{i_2}^{(2)}}}$,$\cdots$,$\frac{\partial{f}}{\partial{\mathbf{u}_{i_N}^{(N)}}}$,$\frac{\partial{f}}{\partial{\T{G}}}$
		
		\STATE $\tilde{x}_\alpha \leftarrow \T{G}\times\{\mathbf{u}\}_{\alpha}$
		
		\FOR{$n = 1,\cdots,N$}
		
		\STATE $\frac{\partial{f}}{\partial{\mathbf{u}_{i}^{(n)}}}
		\leftarrow
		-\big(x_{\alpha}-\tilde{x}_\alpha\big)\big{\lbrack}(\T{G}\times_{-n}\{\mathbf{u}\}_{\alpha})_{(n)}\big{\rbrack}^{\mathsf{T}}+\frac{\lambda_{reg}}{|\Omega^{n,i_n}_{\T{X}}|}\mathbf{u}^{(n)}_{i_n}$
		
		\ENDFOR
		
		\STATE $\frac{\partial{f}}{\partial{\T{G}}}
		\leftarrow
		-\big(x_{\alpha}-\tilde{x}_\alpha\big)\times\{\mathbf{u}\}^{\mathsf{T}}_{\alpha}+\frac{\lambda_{reg}}{|\Omega_{\T{X}}|}\T{G}$
		
		\RETURN $\frac{\partial{f}}{\partial{\mathbf{u}_{i_1}^{(1)}}}$,$\frac{\partial{f}}{\partial{\mathbf{u}_{i_2}^{(2)}}}$,$\cdots$,$\frac{\partial{f}}{\partial{\mathbf{u}_{i_N}^{(N)}}}$,$\frac{\partial{f}}{\partial{\T{G}}}$
		
	\end{algorithmic}
\end{algorithm}

There is no extra time required for calculating $\T{S}$ because $\T{S}$ is generated while calculating $\tilde{x}_\alpha$. Lemma \ref{lemma: Intermediate Sum} shows that $\tilde{x}_\alpha$ is calculated by summing all entries of $\T{S}$.

\begin{lemma} \label{lemma: Intermediate Sum}
	For a given tensor index $\alpha$, estimated tensor entry $\tilde{x}_\alpha=\sum_{j_1=1}^{J_1}\sum_{j_2=1}^{J_2}\cdots\sum_{j_N=1}^{J_N}s_{j_1j_2\cdots j_N}$.
\end{lemma}

\begin{proof}
	The proof is straightforward by Equation \eqref{eqn:hosvd elementwise}.
\end{proof}

We use $\T{S}$ to calculate gradients efficiently.

\begin{definition}\label{def:collapse}
	\textbf{(Collapse)}
	The \textit{Collapse} operation of the intermediate tensor $\T{S}$ on the $n$-th mode outputs a row vector defined by\\
	\centerline{$Collapse(\T{S},n)=\big[\sum_{\forall\delta\in\Omega_{\T{S}}^{n,1}}s_\delta,\sum_{\forall\delta\in\Omega_{\T{S}}^{n,2}}s_\delta, \cdots ,\sum_{\forall\delta\in\Omega_{\T{S}}^{n,J_n}}s_\delta\big]$}
\end{definition}

\textit{Collapse} operation aggregates the elements of intermediate tensor $\T{S}$ with respect to a fixed mode.
We re-express the calculation of gradients for tensor factors in Equations \eqref{eqn:gradient} in an efficient manner.

\begin{lemma} \label{lemma: new gradient calculation}
	\textbf{(Efficient Gradient Calculation)} Followings are equivalent calculations of tensor factors gradients as Equations \eqref{eqn:gradient}.
	\begin{equation}\label{eqn:estimated entry}
	\tilde{x}_\alpha\leftarrow\sum_{j_1=1}^{J_1}\sum_{j_2=1}^{J_2}\cdots\sum_{j_N=1}^{J_N}s_{j_1j_2\cdots j_N}
	\end{equation}
	\begin{equation}\label{eqn:new factor gradient}
	\frac{\partial{f}}{\partial{\mathbf{u}_{i_n}^{(n)}}}\leftarrow-(x_\alpha-\tilde{x}_\alpha)\cdot Collapse(\T{S},n)\oslash \mathbf{u}_{i_n}^{(n)}+\frac{\lambda_{reg}}{|\Omega_{\T{X}}^{n,i_n}|}\mathbf{u}_{i_n}^{(n)}
	\end{equation}
	\begin{equation}\label{eqn:new core gradient}
	\frac{\partial{f}}{\partial{\T{G}}}\leftarrow-(x_\alpha-\tilde{x}_\alpha)\cdot{\T{S}}\oslash{\T{G}}+\lambda_{reg}\T{G}
	\end{equation}
	where $\alpha=(i_1i_2\cdots i_N)$ and $\oslash$ is element-wise division.
\end{lemma}

\begin{proof}
	In Lemma \ref{lemma: Intermediate Sum}, Equation \eqref{eqn:estimated entry} is proved. To prove the equivalence of Equation \eqref{eqn:new factor gradient} and the first equation of Equations \eqref{eqn:gradient}, it suffices to show $[(\T{G}\times_{-n}\{\mathbf{u}\}_{\alpha})_{(n)}]^{\mathsf{T}}=Collapse(\T{S},n)\oslash \mathbf{u}_{i_n}^{(n)}$
	where $\alpha=(i_1\cdots i_N)\in\Omega_{\T{X}}$
	and $\delta=(j_1\cdots j_N)\in\Omega_{\T{G}}^{n,k}$. We use Equation \eqref{eqn:n mode except row set product} for the proof.
	\begin{equation*}
	[(\T{G}\times_{-n}\{\mathbf{u}\}_{\alpha})_{(n)}]^{\mathsf{T}}_{k}=\sum_{\forall\delta\in\Omega_{\T{G}}^{n,k}}{g_{\delta}u^{(1)}_{i_1j_1}\cdots u^{(n-1)}_{i_{n-1}j_{n-1}}u^{(n+1)}_{i_{n+1}j_{n+1}}\cdots u^{(N)}_{i_{N}j_{N}}}
	\end{equation*}
	\begin{equation*}
	=\sum_{\forall\delta\in\Omega_{\T{G}}^{n,k}}{g_{\delta}u^{(1)}_{i_1j_1}\cdots u^{(n-1)}_{i_{n-1}j_{n-1}}u^{(n)}_{i_nk}u^{(n+1)}_{i_{n+1}j_{n+1}}\cdots u^{(N)}_{i_{N}j_{N}}}\big/u^{(n)}_{i_nk}
	\end{equation*}
	\begin{equation*}
	=\sum_{\forall\delta\in\Omega_{\T{S}}^{n,k}}s_\delta/u^{(n)}_{i_nk}
	=\frac{[Collapse(\T{S},n)]_k}{u^{(n)}_{i_nk}}=[Collapse(\T{S},n)\oslash \mathbf{u}_{i_n}^{(n)}]_k
	\end{equation*}
	Next, to show the equivalence of Equation \eqref{eqn:new core gradient} and the second equation of Equations \eqref{eqn:gradient}, it suffices to show $1\times\{\mathbf{u}\}_{\alpha}^{\mathsf{T}}=\T{S}\oslash\T{G}$.
	\begin{equation*}
	[1\times\{\mathbf{u}\}_{\alpha}^{\mathsf{T}}]_{\gamma=(l_1l_2\cdots l_N)}=u^{(1)}_{i_1l_1}u^{(2)}_{i_2l_2}\cdots u^{(N)}_{i_Nl_N}
	\end{equation*}
	\begin{equation*}
	=g_{\gamma}u^{(1)}_{i_1l_1}\cdots u^{(N)}_{i_Nl_N}/g_{\gamma}
	=s_{\gamma}/g_{\gamma}
	=[\T{S}\oslash\T{G}]_{\gamma}
	\end{equation*}
\end{proof}

\begin{algorithm} [t]
	\small
	\caption{\textit{compute\_gradient\_opt}($\alpha$,$x_\alpha$,$\T{G}$)} \label{alg:compute_gradient_new}
	\begin{algorithmic}[1]
		\REQUIRE Tensor entry $x_\alpha$, $\alpha=(i_1\cdots i_N)$$\in\Omega_{\T{X}}$, core tensor $\T{G}$
		
		\ENSURE Gradients $\frac{\partial{f}}{\partial{\mathbf{u}_{i_1}^{(1)}}}$,$\frac{\partial{f}}{\partial{\mathbf{u}_{i_2}^{(2)}}}$,$\cdots$,$\frac{\partial{f}}{\partial{\mathbf{u}_{i_N}^{(N)}}}$,$\frac{\partial{f}}{\partial{\T{G}}}$
		
		\STATE $\tilde{x}_\alpha\leftarrow0$
		
		\FOR {$\forall(j_1j_2\cdots j_N)\in\Omega_{\T{G}}$}
		
		\STATE $s_{j_1j_2\dots j_N}\leftarrow g_{j_1j_2\dots j_N}u_{i_1j_1}^{(1)}u_{i_2j_2}^{(2)}\cdots u_{i_Nj_N}^{(N)}$
		
		\STATE $\tilde{x}_\alpha\leftarrow\tilde{x}_\alpha+s_{j_1j_2\dots j_N}$
		
		\ENDFOR
		
		\FOR{$n = 1,\dots,N$}
		
		\STATE $\frac{\partial{f}}{\partial{\mathbf{u}_{i_n}^{(n)}}}\leftarrow-(x_\alpha-\tilde{x}_\alpha)\cdot Collapse(\T{S},n)\oslash \mathbf{u_{i_n}^{(n)}}+\frac{\lambda_{reg}}{|\Omega_{\T{X}}^{n,i_n}|}\mathbf{u_{i_n}^{(n)}}$
		
		\ENDFOR
		
		\STATE $\frac{\partial{f}}{\partial{\T{G}}}\leftarrow-(x_\alpha-\tilde{x}_\alpha)\cdot{\T{S}}\oslash{\T{G}}+\lambda_{reg}\T{G}$
		
		\RETURN $\frac{\partial{f}}{\partial{\mathbf{u}_{i_1}^{(1)}}}$,$\frac{\partial{f}}{\partial{\mathbf{u}_{i_2}^{(2)}}}$,$\dots$,$\frac{\partial{f}}{\partial{\mathbf{u}_{i_N}^{(N)}}}$,$\frac{\partial{f}}{\partial{\T{G}}}$
		
	\end{algorithmic}
\end{algorithm}

\vspace*{-0.3cm}
\method-opt replaces \textit{compute\_gradient} (Algorithm \ref{alg:compute_gradient}) of \method-naive with \textit{compute\_gradient\_opt} (Algorithm \ref{alg:compute_gradient_new}), a time-improved alternative using Lemma \ref{lemma: new gradient calculation}.
We prove that the new calculation scheme is faster than the previous one.
\vspace*{-0.2cm}
\begin{lemma} \label{lemma:time comparison of new update}
	\textit{compute\_gradient\_opt} is faster than \textit{compute\_gradient}. The time complexity of \textit{compute\_gradient} is $\mathcal{O}(N^2J^N)$ and the time complexity of \textit{compute\_gradient\_opt} is $\mathcal{O}(NJ^N)$ where $J_1=J_2=\cdots=J_N=J$.
\end{lemma}
\vspace*{-0.2cm}
\begin{proof}
	We assume that $I_1=I_2=\cdots=I_N=I$ for brevity.	First, we calculate the time complexity of \textit{compute\_gradient} (Algorithm \ref{alg:compute_gradient}).
	Given a tensor index $\alpha$, computing $\tilde{x}_{\alpha}$ (line 1 of Algorithm \ref{alg:compute_gradient}) takes $\mathcal{O}(NJ^N)$.
	Computing ($\T{G}\times_{-n}\{\mathbf{u}\}_{\alpha}$) (line 3) takes $\mathcal{O}(NJ^N)$. Thus, aggregate time for calculating the row gradient for all modes (lines 2-4) takes $\mathcal{O}(N^2J^N)$. Calculating $(x_{\alpha}-\tilde{x}_{\alpha})\times\{\mathbf{u}\}_{\alpha}^{\mathsf{T}}$ (line 5) takes $\mathcal{O}(NJ^N)$. In sum, \textit{compute\_gradient} takes $\mathcal{O}(N^2J^N)$ time.
	Next, we calculate the time complexity of \textit{compute\_gradient\_opt} (Algorithm \ref{alg:compute_gradient_new}). Computing an entry of intermediate data $\T{S}$ (line 3 of Algorithm \ref{alg:compute_gradient_new}) takes $\mathcal{O}(N)$. Aggregate time for getting $\T{S}$ (lines 2-5) is $\mathcal{O}(NJ^N)$ because $|\Omega_{\T{G}}|=\mathcal{O}(J^N)$.
	Calculating row gradient for all modes (lines 6-8) takes $\mathcal{O}(NJ^N)$ because $Collapse$ operation takes $\mathcal{O}(J^N)$.
	Calculating gradient for core tensor (line 9) takes $\mathcal{O}(J^N)$. In sum, \textit{compute\_gradient\_opt} takes $\mathcal{O}(NJ^N)$ time.
\end{proof}
\vspace*{-0.2cm}

\begin{table}[h]
	\small
	\setlength{\tabcolsep}{1pt}
	\caption{Comparison of time complexity (per iteration) and memory usage of our proposed $\mathbf{S^3}$CMTF and other CMTF algorithms. $\mathbf{S^3}$CMTF-opt shows the lowest time complexity and $\mathbf{S^3}$CMTF-naive shows the lowest memory usage. For simplicity, we assume that all modes are of size $I$, of rank $J$, and an $I\times K$ matrix is coupled to one mode. $P$ is the number of parallel cores. (* indicates the lowest time or memory.)}
	\begin{center}
		{
			\begin{tabular}{L{2.3cm} L{3.7cm} L{2cm}}
				\toprule
				& \textbf{Time complexity} (per iter.) & \textbf{Memory usage} \\
				
				\midrule
				
				\textbf{$\mathbf{S^3}$CMTF-naive}& $\mathcal{O}(|\Omega_{\T{X}}|N^2J^{N}/P+|\Omega_{\mathbf{Y}}|J/P)$ & ${\mathcal{O}(PJ)}$*  \\
				
				\textbf{$\mathbf{S^3}$CMTF-opt}& ${\mathcal{O}(|\Omega_{\T{X}}|NJ^{N}/P+|\Omega_{\mathbf{Y}}|J/P)}$* & ${\mathcal{O}(PJ^N)}$  \\
				
				CMTF-Tucker-ALS &
				$\mathcal{O}(NI^{N-1}J^2+NI^2J^{N-1}+I^2K)$ &
				$\mathcal{O}(IJ^{N-1})$  \\
				
				CMTF-OPT & $\mathcal{O}(|\Omega_{\T{X}}|NJ+NI^{N-1}J+IJK)$ & $\mathcal{O}(I^{N-1}J+JK)$  \\
				
				\bottomrule
				
			\end{tabular}
		}
	\end{center}
	\label{tab:comparison}
\end{table}

\subsection{Analysis}\label{subsec:analysis}

We analyze the proposed method in terms of time complexity per iteration. For simplicity, we assume that $I_1=I_2=\cdots=I_N=I$, and $J_1=J_2=\cdots=J_N=J$.
Table \ref{tab:comparison} summarizes the time complexity (per iteration) and memory usage of \method and other methods.
Note that the memory usage refers to the auxiliary space for temporary variables used by a method.
\vspace*{-0.2cm}

\begin{lemma}\label{lemma:time complexity}
	The time complexity (per iteration) of \:\method-naive is $\mathcal{O}(|\Omega|N^2J^{N}/P+|\Omega_{\mathbf{Y}}|J/P)$ and the time complexity (per iteration) of \:\method-opt is \:$\mathcal{O}(|\Omega|NJ^{N}/P+|\Omega_{\mathbf{Y}}|J/P)$
	where $P$ denotes the number of parallel cores.
\end{lemma}
\vspace*{-0.3cm}

\begin{proof}
	First, we check the time complexity of \method-naive (Algorithm \ref{alg:s3cmtf}). When a tensor index $\alpha$ is picked in the inner loop (line 4 of Algorithm \ref{alg:s3cmtf}), calculating gradients with respect to tensor factors (line 5) takes $\mathcal{O}(N^2J^N)$ as shown in Lemma \ref{lemma:time comparison of new update}.
	Updating factor rows (line 6) takes $\mathcal{O}(NJ)$, and updating core tensor (line 7) takes $\mathcal{O}(J^N)$. If a coupled matrix index $\beta$ is picked (line 9), calculating $\tilde{y}_{\beta}$ (line 10) takes $\mathcal{O}(J)$. Calculating and updating the factor rows corresponding to coupled matrix entry (lines 10-12) take $\mathcal{O}(J)$. All calculations except updating core tensor (line 7) are conducted in parallel.
	Finally, for all $\alpha\in\Omega_{\T{X}}$ and $\beta\in\Omega_{\mathbf{Y}}$, \method-naive takes $\mathcal{O}(|\Omega_{\T{X}}|N^2J^{N}/P+|\Omega_{\mathbf{Y}}|J/P)$ for one iteration. 
	\method-opt uses \textit{compute\_gradient\_opt} instead of \textit{compute\_gradient} in line 5 of Algorithm \ref{alg:s3cmtf}, whose time complexity is shown in Lemma \ref{lemma:time comparison of new update}.
	Overall running time per iteration for \method-opt is $\mathcal{O}(|\Omega_{\T{X}}|NJ^{N}/P+|\Omega_{\mathbf{Y}}|J/P)$.
\end{proof} 
\vspace*{-0.2cm}

\vspace*{-0.1cm}
\section{Experiments}
    \label{sec:experiments}
    In this and the next sections, we experimentally evaluate \method. Especially, we answer the following questions.

\noindent\textbf{Q1} : \textbf{Performance (Section \ref{subsec:performance})} How accurate and fast is \method compared to competitors?\\
\textbf{Q2} : \textbf{Scalability (Section \ref{subsec:scalability})} How do \method and other methods scale in terms of dimension, the number of observed entries, and the number of cores?\\
\textbf{Q3} : \textbf{Discovery (Section \ref{sec:case})} What are the discoveries of applying \method on real-world data?

\begin{table}[t!]
	\small
	\setlength{\tabcolsep}{1pt}
	\caption{Summary of the data used for experiments. K: thousand, and M: million. Data of density 1 are fully observed.}
	\begin{center}
		{
			\begin{tabular}{L{1.3cm} L{2.3cm} R{1.9cm} R{1.4cm} R{1.1cm}}
				\toprule
				\textbf{Name} & \textbf{Data} & \textbf{Dimensionality} & \textbf{\# entries} & \textbf{Density} \\
				
				\midrule
				
				MovieLens & User-Movie-Time & 71K-11K-157 & 10M & $\sim$$10^{-4}$  \\
				& Movie-Genre & 20 & 214K & 1\\
				Netflix & User-Movie-Time & 480K-18K-74 & 100M & $\sim$$10^{-4}$  \\
				& Movie-Yearmonth & 110 & 2M & 1\\
				Yelp & User-Business-Time & 1M-144K-149 & 4M & $\sim$$10^{-7}$  \\
				& User-User  & 1M & 7M & $\sim$$10^{-4}$ \\
				& Business-Category & 1K & 172M & 1\\
				& Business-City & 1K & 126M & 1\\
				Synthetic & 3-mode tensor & 1K$\sim$100M & 1K$\sim$100M & $10^{-20\sim-3}$ \\
				& Matrix & 1K$\sim$100M & 1K$\sim$100M & $10^{-11\sim-4}$ \\
				
				\bottomrule
				
			\end{tabular}
		}
	\end{center}
	\label{tab:data}
	\vspace*{-0.1cm}
\end{table}
\vspace{-3mm}

\subsection{Experimental Settings}
\textbf{Data.} Table \ref{tab:data} shows the data we used in our experiments.
We use three real-world datasets (MovieLens\footnote{http://grouplens.org/datasets/movielens/10m}, Netflix\footnote{http://www.netflixprize.com}, and Yelp\footnote{http://www.yelp.com/dataset\_challenge}) and generate synthetic data to evaluate \method.
Each entry of the real-world datasets represents a rating, which consists of (user, \lq item\rq, time; rating) where \lq item\rq~ indicates \lq movie\rq~ for MovieLens and Netflix, and \lq business\rq~ for Yelp.
We use (movie, genre) and (movie, year) as coupled matrices for MovieLens and Netflix, respectively.
We use (user, user) friendship matrix, (business, category) and (business, city) matrices for Yelp.
We generate 3-mode synthetic random tensors with dimensionality $I$ and corresponding coupled matrices.
We vary $I$ in the range of 1K$\sim$100M and the number of tensor entries in the range of 1K$\sim$100M.
We set the number of entries as $|\Omega_{\mathbf{Y}}|=\frac{1}{10}|\Omega_{\T{X}}|$ for synthetic coupled matrices.

\textbf{Measure.} We use test RMSE as the measure for tensor reconstruction error.
\begin{equation*}
\text{test RMSE}=\sqrt{\frac{1}{|\Omega_{test}|}\sum_{\forall\alpha\in\Omega_{test}}{(x_\alpha-\tilde{x}_{\alpha})^2}}
\end{equation*}

where $\Omega_{test}$ is the index set of the test tensor, $x_\alpha$ represents each test tensor entry, and $\tilde{x}_{\alpha}$ is the corresponding reconstructed value. 

\textbf{Methods.} We compare \method-naive and \method-opt with other single machine CMTF methods: CMTF-Tucker-ALS and CMTF-OPT (described in Section \ref{subsec:CMTF}). To examine multi-core performance, we run two versions of \method-opt: \method-opt1 (1 core), and \method-opt20 (20 cores).
We exclude distributed CMTF methods \cite{jeon2016scout,jeon2015haten2,beutel2014flexifact} because they are designed for Hadoop with multiple machines, and thus take too much time for single machine environment.
For example, \cite{oh2017s} reported that HaTen2 \cite{jeon2015haten2} takes 10,700s to decompose 4-way tensor with $I=10K$ and $|\Omega_{\T{X}}|=100K$, which is almost 7,000$\times$ slower than a single machine implementation of \method-opt.
%
For CMTF-Tucker-ALS, we use a MATLAB implementation based on Tucker-MET \cite{kolda2008scalable}.
For CMTF-OPT, we use MATLAB implementation of CMTF Toolbox 1.1\footnote{http://www.models.life.ku.dk/joda/CMTF\_Toolbox}.
We implement \method with C++, and OpenMP library for multi-core parallelization.
We note for fair comparison that a fully optimized C++ implementation might be faster than MATLAB implementation for loop-oriented algorithms; on the other hand, MATLAB potentially beats C++ on matrix and array calculations due to its high-degree optimization and auto multi-core calculations.
Regardless of the implementation environment, however, 
our main contributions still holds: \method scales to large data and a number of cores with high accuracy thanks to the careful use of intermediate data,
while competitors fail with out-of-memory error due to their excessive memory usage.

We conduct all experiments on a machine equipped with Intel Xeon E5-2630 v4 2.2GHz CPU and 256GB RAM.
All parameters are set to the best found values.
We mark out-of-memory (O.O.M.) error when the memory usage exceeds the limit and out-of-time (O.O.T.) error when the iteration time exceeds $10^4$ seconds.

\textbf{Parameters.} We set pre-defined parameters: tensor rank $J$, regularization factor $\lambda_{reg}$, $\lambda_{m}$, the initial learning rate $\eta_0$, and decay rate $\mu$. We set $\lambda_{reg}$ to 0.1, $\lambda_{m}=10$, and $\mu=0.1$ for all datasets. For rank and initial learning rate, MovieLens: $J=12, \eta=0.001$, Netflix: $J=11, \eta=0.001$, and Yelp: $J=10, \eta=0.0005$.


\subsection{Performance of $\mathbf{S^{3}CMTF}$} \label{subsec:performance}
\vspace*{-0.1cm}
We measure the performance of \method to answer Q1. As seen in Figure \ref{fig:Time Error} and \ref{fig:IterTime}, \method improves the test error of existing methods by 2.1$\sim$4.1$\times$ and decreases the running time for one iteration by 11$ \sim$43$\times$. The details of the experiments are as follows.

\textbf{Accuracy.}
We divide each data tensor into 80\%/20\% for train/test sets. 
The lower error for a same elapsed time implies the better accuracy and faster convergence.
Figure \ref{fig:Time Error} shows the changes of test RMSE of each method on three datasets over elapsed time which are the answers for Q1. \method achieves the lowest error compared to others for the same elapsed time.
For Netflix and Yelp, CMTF-Tucker-ALS shows O.O.M. error.
On MovieLens, the best error of competitors is 2.904 of CMTF-OPT. In the same elapsed time, \method-opt20 achieves $3.6\times$ lower error, 0.8037.
For Netflix, we improve the error of CMTF-OPT (3.764) by $4.1\times$ to achieve 0.9147.
In Yelp, the best error of CMTF-OPT is 2.663.
\method-opt20 shows the lowest error of 1.253 in a few tens of iterations, and after then, it falls into an over-fitting zone. \method-opt20 achieves $2.1\times$ less error than the best of CMTF-OPT.

\vspace*{-0.1cm}
\begin{figure} [!h]
	\begin{center}
		\includegraphics[width=0.5 \textwidth]{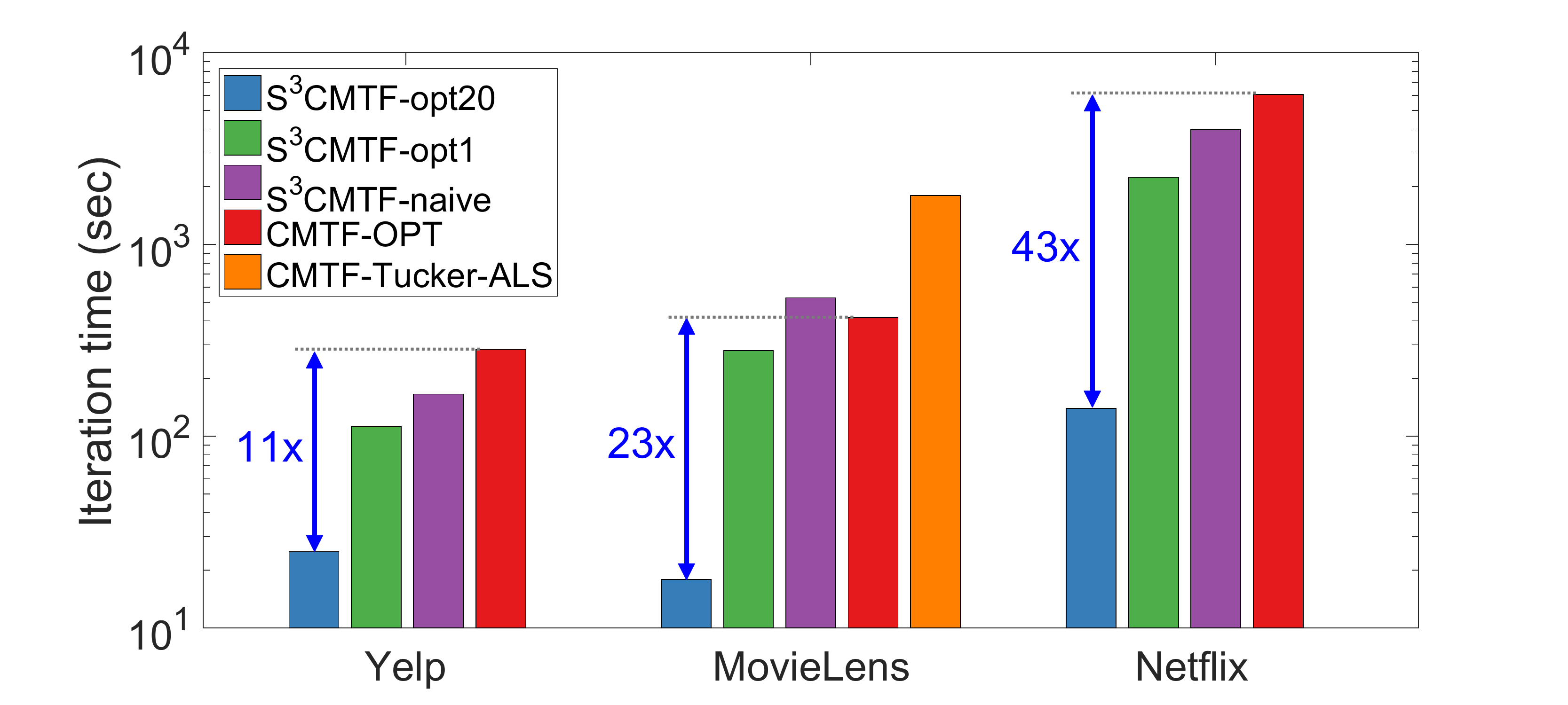}
	\end{center}
	\vspace*{-0.1cm}
	\caption{
		Running time of each method for one iteration.
		$\mathbf{S^{3}CMTF}$-opt20 is 11$\sim$43$\times$ faster than existing methods.}
	\label{fig:IterTime}
\end{figure}
\vspace*{-0.1cm}

\begin{figure*} [t]
	\includegraphics[width=1 \textwidth]{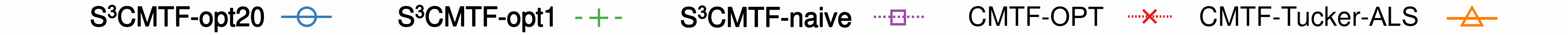}
	\subfloat[\textbf{Running time vs. Dimensionality}]
	{	\includegraphics[width=0.30 \textwidth]{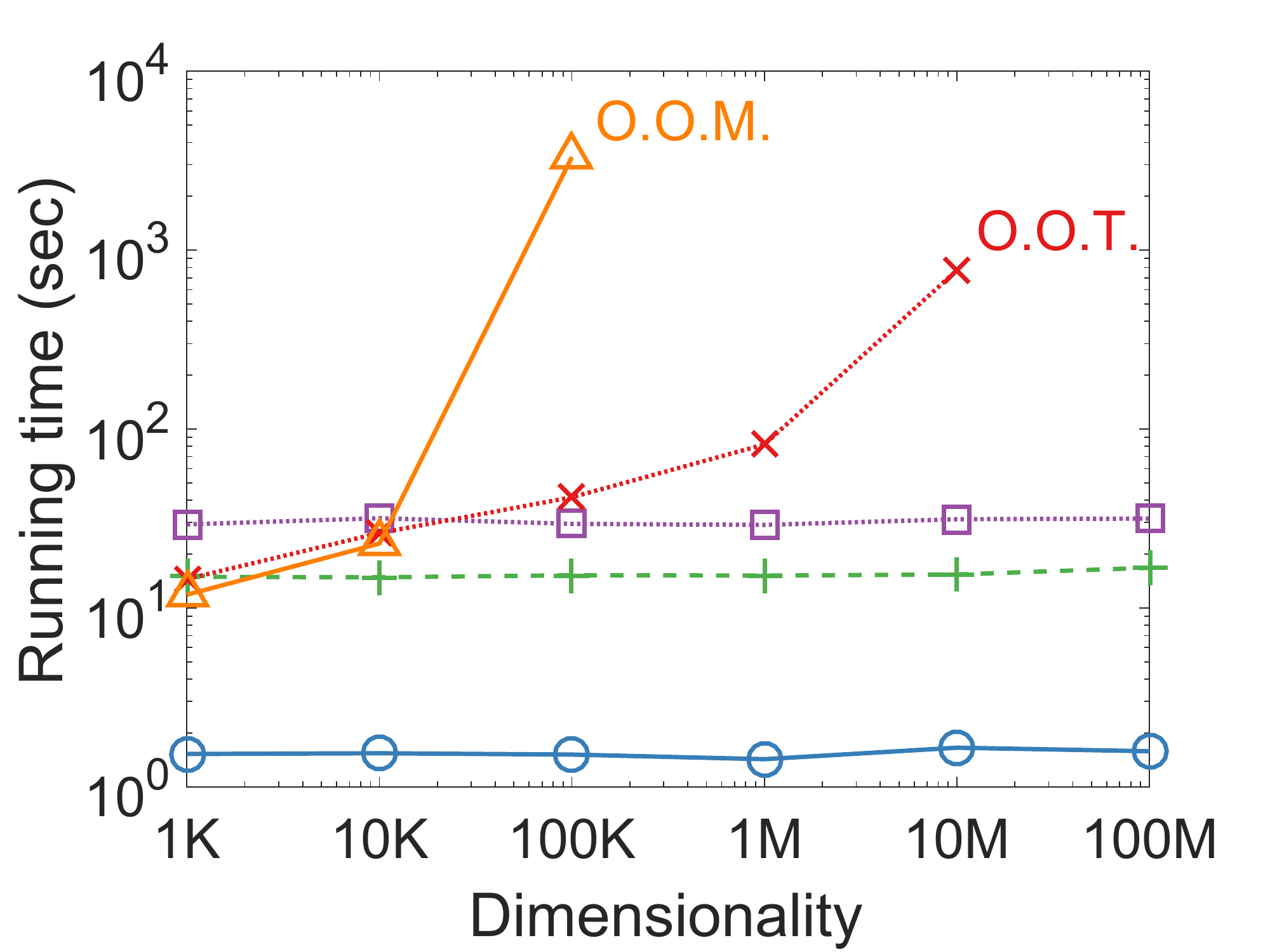}
	}
	\subfloat[\textbf{Running time vs. Number of entries}]
	{	\includegraphics[width=0.30 \textwidth]{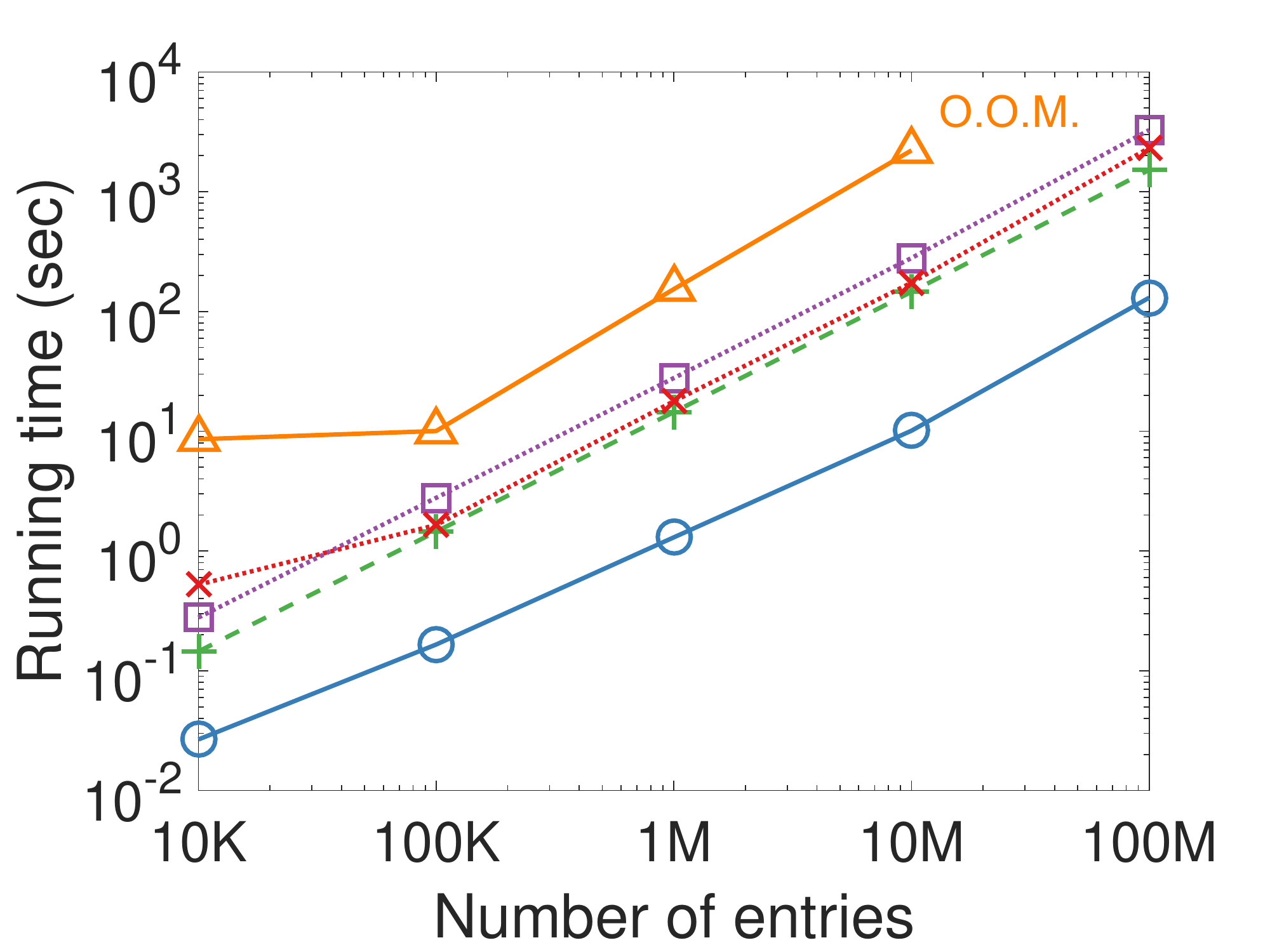}
	}
	\subfloat[\textbf{Parallel scalability}]
	{	\includegraphics[width=0.30 \textwidth]{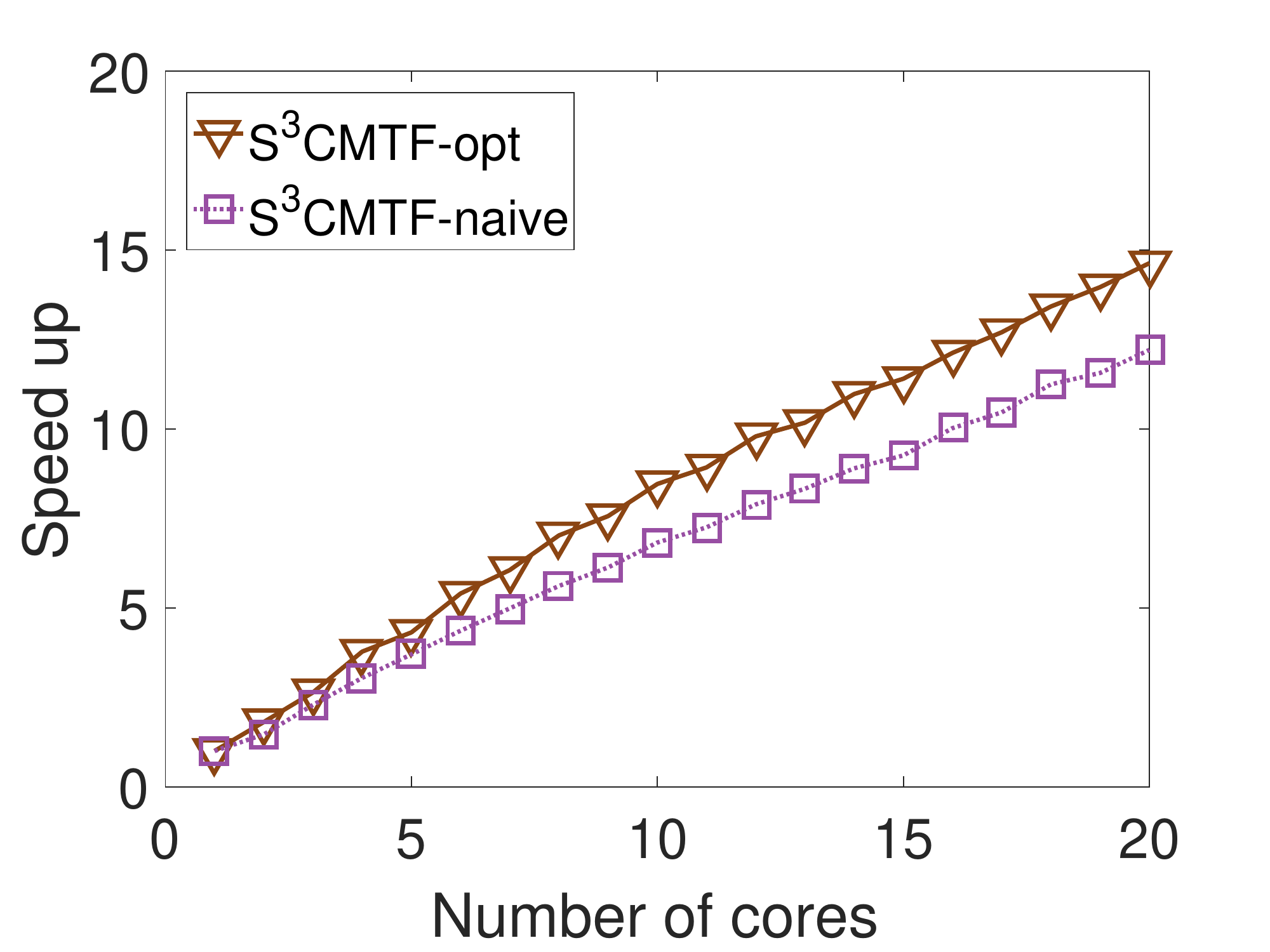}
	}
	\vspace*{-0.1cm}
	\caption{Comparison of scalability. (a) $\mathbf{S^{3}CMTF}$ takes constant time as dimensionality grows with the fixed number of entries.
		(b) $\mathbf{S^{3}CMTF}$ shows linear scalability as the number of entries increases.
		(c) $\mathbf{S^{3}CMTF}$-naive and $\mathbf{S^{3}CMTF}$-opt show linear \textit{Speed up} as the number of cores grows.
		O.O.M.: out of memory error, O.O.T.: out of time error.
	}
	\label{fig:Scalability}
	\vspace*{-0.2cm}
\end{figure*}

\textbf{Running time.}
We empirically show that \method achieves the best speed in terms of running time.
Figure \ref{fig:IterTime} shows the average running time of each method on the three data. \method-opt20 improves the running time of the best competitor by more than an order of magnitude for all datasets.
In Yelp, \method-opt20 takes 25s for an iteration which is $11 \times$ faster than 283s of CMTF-OPT.
In MovieLens, \method-opt20 takes 18s, $23 \times$ faster compared to 415s of CMTF-OPT.
For Netflix, \method-opt20 achieves $43 \times$ faster running time (140s) compared to that of CMTF-opt (6,100s). Note that CMTF-Tucker-ALS shows O.O.M. error for all data except for MovieLens.
Though \method-naive and \method-opt1 show comparable running times to that of CMTF-OPT for an iteration, they converge faster and are more accurate as shown in Figure \ref{fig:Time Error} since they capture inter-relations between factors with higher model capacities.

\begin{figure} [t]
	\vspace*{-0.3cm}
	\begin{center}
		\subfloat[\textbf{Convergence comparison}]
		{	\includegraphics[width=0.24 \textwidth]{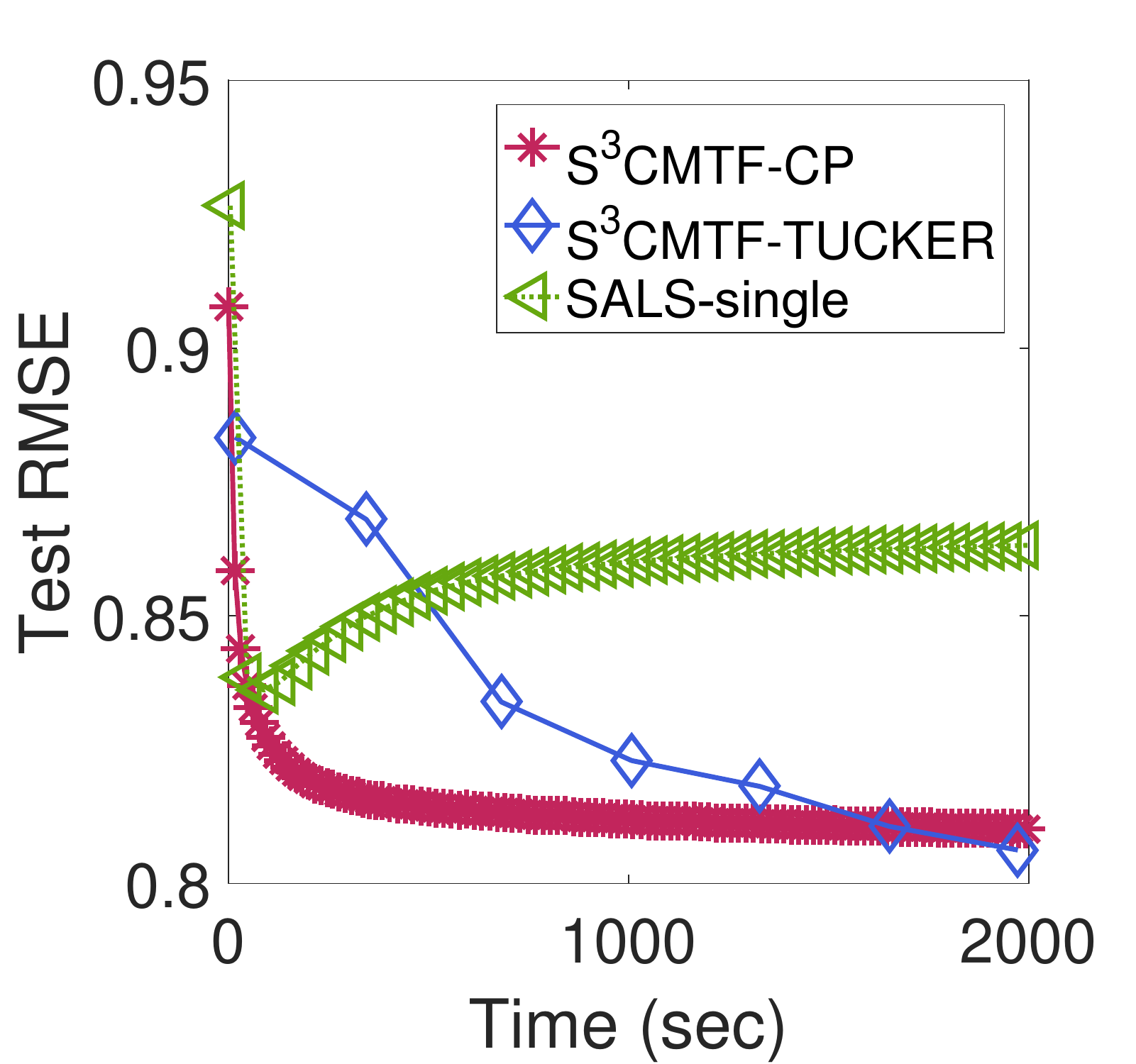}
		}
		\subfloat[\textbf{Time comparison}]
		{	\includegraphics[width=0.24 \textwidth]{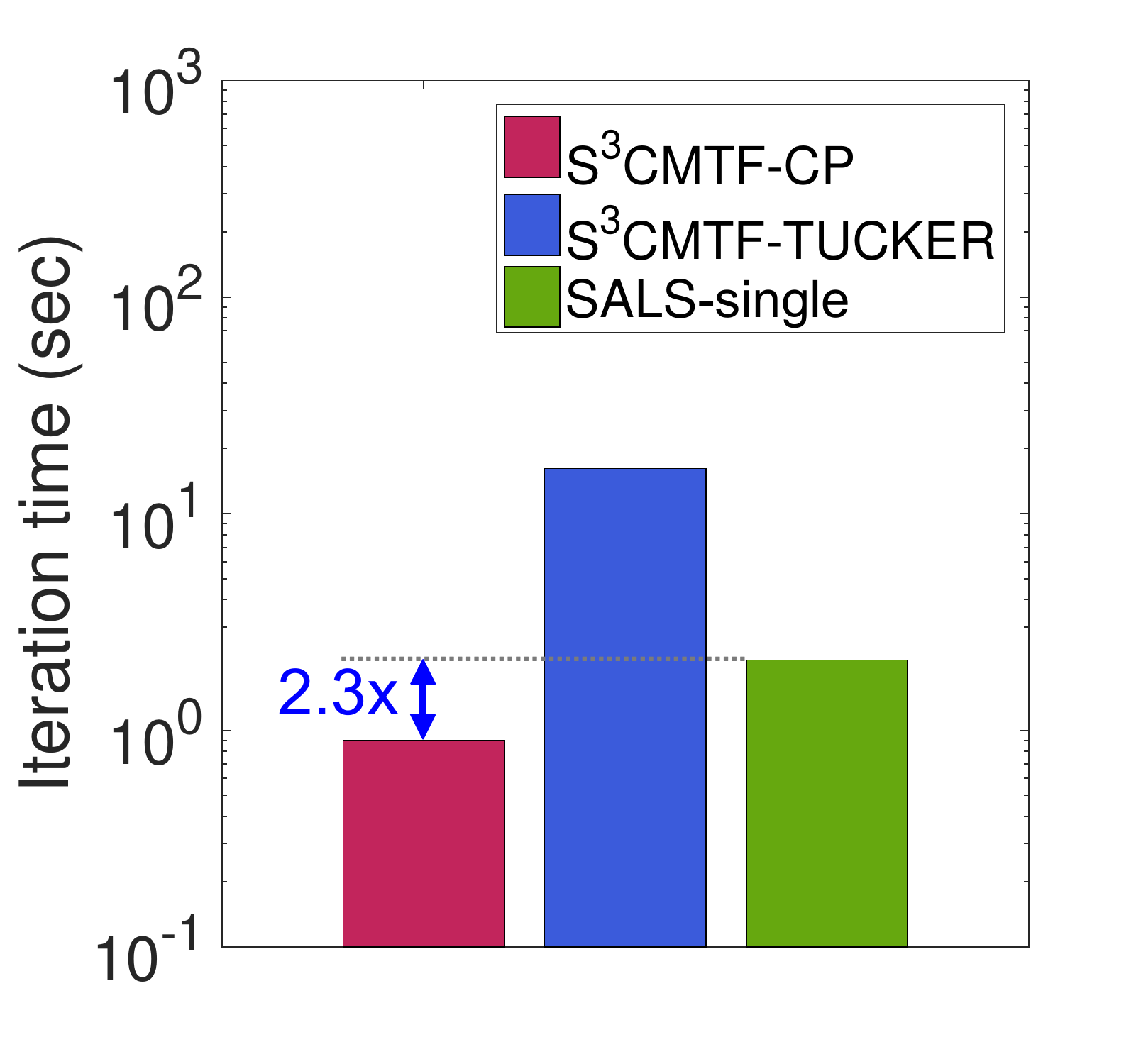}
		}
		\caption{Comparison with SALS-single. We compare two non-coupled version of $\mathbf{S^{3}CMTF}$, $\mathbf{S^{3}CMTF}\mathbf{-CP}$ and $\mathbf{S^{3}CMTF}\mathbf{-TUCKER}$ with the parallel CP decomposition method, SALS-single. For (a), we set 1 mark per 20 iterations for clarity.
			(a) $\mathbf{S^{3}CMTF}\mathbf{-CP}$ and $\mathbf{S^{3}CMTF}\mathbf{-TUCKER}$ converge to lower test RMSE than SALS-single while SALS-single overfits after few decades of iterations. Note that $\mathbf{S^{3}CMTF}\mathbf{-TUCKER}$ finds lower Test RMSE compared to the other methods.
			(b) $\mathbf{S^{3}CMTF}\mathbf{-CP}$ is 2.3$\times$ faster than SALS-single.
		}
		\label{fig:sals comp}
	\end{center}
\end{figure}
We compare our method with the multi-core version of SALS-single \cite{DBLP:journals/tkde/ShinSK17}, a CP decomposition algorithm, to demonstrate the high performance of \method compared to up-to-date decomposition algorithms. We implement CP version of our method, \method-CP, by setting $\T{G}$ to be hyper-diagonal.
Since CMTF is the extended problem of tensor decomposition, \method is used for tensor decomposition in a straightforward way by not coupling any matrices.
\method-TUCKER denote the non-coupled version of \method-opt. MovieLens tensor is used for decomposition.
Figure \ref{fig:sals comp} shows that \method is better than SALS-single in terms of both error and time.
\vspace*{-0.4cm}
\subsection{Scalability Analysis} \label{subsec:scalability}
We inspect scalability of our proposed method and others to answer Q2, in terms of two aspects: data scalability and parallel scalability. We use synthetic data of varying size for evaluation. As a result, we show the running time (for one iteration) of \method follows our theoretical analysis in Section \ref{subsec:analysis}.

\textbf{Data Scalability.}
The time complexity of CMTF-Tucker-ALS and CMTF-OPT have $\mathcal{O}(NI^{N-1}J^2)$ and $\mathcal{O}(NI^{N-1}J)$ as their dominant terms, respectively.
In contrast, \method exploits the sparsity of input data, and has the time complexity linear to the number of entries ($|\Omega_{\T{X}}|$, $|\Omega_{\mathbf{Y}}|$) and independent to the dimensionality ($I$) as shown in Lemma \ref{lemma:time complexity}.
Figures \ref{fig:Scalability}a and \ref{fig:Scalability}b show that the running time (for one iteration) of \method follows our theoretical analysis in Section \ref{subsec:analysis}.

First, we fix $|\Omega_{\T{X}}|$ to 1M and $|\Omega_{\mathbf{Y}}|$ to 100K, and vary dimensionality $I$ from 1K to 100M.
Figure \ref{fig:Scalability}a shows the running time (for one iteration) of all methods. Note that all our proposed methods achieve constant running time as dimensionality increases because they exploit the sparsity of data by updating factors related to only observed data entries.
However, CMTF-Tucker-ALS shows O.O.M. when $I\geq10M$, and CMTF-OPT presents O.O.T. when $I=100M$.
Next, we investigate the data scalability over the number of entries. We fix $I$ to 10K and raise $|\Omega_{\T{X}}|$ from 10K to 100M.
CMTF-Tucker-ALS shows O.O.M. when $|\Omega_{\T{X}}|=100M$, and CMTF-OPT shows near-linear scalability.
Focusing on the results of \method, all three versions of our approach show linear relation between running time and $|\Omega_{\T{X}}|$.

\textbf{Parallel Scalability.}
We conduct experiments to examine parallel scalability of \method on shared memory systems. For measurement, we define \textit{Speed up} as \textit{(Iteration time on 1 core)}/\textit{(Iteration time)}.
Figure \ref{fig:Scalability}c shows the linear \textit{Speed up} of \method-naive and \method-opt.
\method-opt earns higher \textit{Speed up} than \method-naive because it reduces reading accesses for core tensor by utilizing intermediate data.


\vspace*{-0.1cm}
\section{Discovery}
	\label{sec:case}
	
\begin{table}[tbp]
	\small
	\setlength{\tabcolsep}{1pt}
	\caption{Clustering results on business factor $\mathbf{U}^{(2)}$ found by $\mathbf{S^3}$CMTF. We found dominant spatial and categorical characteristics from each cluster. Businesses in a same cluster tend to be in adjacent cities and are included in similar categories.}
	\begin{center}
		{
			\begin{tabular}{L{1.0cm} L{2.1cm} L{4.9cm}}
				\toprule
				\textbf{Cluster} & \textbf{Location / \newline Category} & \textbf{Top-10 Businesses} \\
				\midrule
				\textbf{C1} & Las Vegas, US/ \newline Travel \& Entertainment & Nocturnal Tours, Eureka Casino, Happi Inn, Planet Hollywood Poker Room, Circus Midway Arcade, etc. \\
				\midrule
				\textbf{C2} & Arizona, US/ \newline Real estate \& Home services & ENMAR Hardwood Flooring, Sprinkler Dude LLC, Eklund Refrigeration, NR Quality Handyman, The Daniel Montez Real Estate Group, etc.  \\
				\midrule
				\textbf{C11} & Ontario, Canada/\newline Restaurants \& Deserts & Jyuban Ramen House, Tim Hortons, Captain John Donlands Fish and Chips, Cora's Breakfast \& Lunch, Pho Pad Thai, etc.
				\\
				\midrule
				\textbf{C17} & Ohio, US/\newline Food \& Drinks & ALDI, Pulp Juice and Smoothie Bar, One Barrel Brewing, Wok N Roll Food Truck, Gas Pump Coffee Company, etc.
				\\
				\bottomrule
				
			\end{tabular}
		}
	\end{center}
	\label{tab:cluster category}
	\vspace*{-0.2cm}
\end{table}
\vspace*{-0.1cm}
\begin{figure} [tb]
	\begin{center}
		\includegraphics[width=0.5 \textwidth]{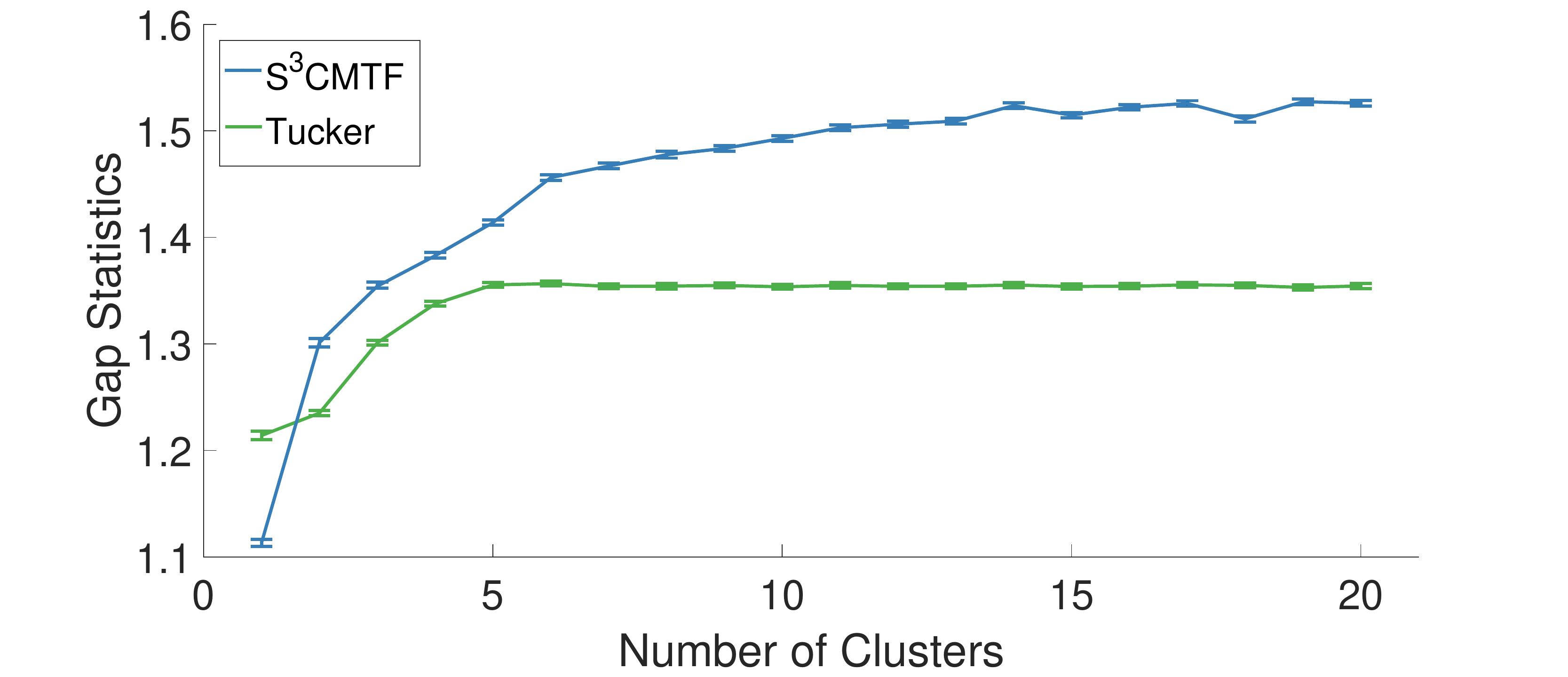}
	\end{center}
	\vspace*{-0.1cm}
	\caption{
		Gap statistics on $\mathbf{U}^{(2)}$ of $\mathbf{S^3}$CMTF and the Tucker decomposition for Yelp dataset. $\mathbf{S^3}$CMTF outperforms the naive Tucker decomposition for its clustering ability.} 
	\label{fig:gap statistics}
\end{figure}
In this section, we use \method for mining real-world data, Yelp, to answer the question Q3 in the beginning of Section \ref{sec:experiments}.
First, we demonstrate that \method has better discernment for business entities compared to the naive decomposition method by jointly capturing spatial and categorical prior knowledge.
Second, we show how \method is possibly applied to the real recommender systems. It is an open challenge to jointly capture the spatio-temporal context along with user preference data \cite{gao2013exploring}. We exemplify a personal recommendation for a specific user.
For discovery, we use the total Yelp data tensor along with coupled matrices as explained in Table~\ref{tab:data}. For better interpretability, we found non-negative factorization by applying projected gradient method \cite{lin2007projected}.
Orthogonality is not applied to keep non-negativity, and each column of factors is normalized.

\textbf{Cluster Discovery.}
First, we compare discernment by \method and the Tucker decomposition. We use the business factor $\mathbf{U}^{(2)}$. Figure \ref{fig:gap statistics} shows gap statistic values of clustering business entities with k-means clustering algorithm.
Higher gap statistic value means higher clustering ability \cite{tibshirani2001estimating}, thus \method outperforms the Tucker decomposition for entity clustering.

As the difference between \method and the Tucker decomposition is the existence of coupled matrices, the high performance of \method is attributed to the unified factorization using spatial and categorical data as prior knowledge. Table \ref{tab:cluster category} shows the found clusters of business entities. Note that each cluster represents a certain combination of spatial and categorical characteristics of business entities. 

\textbf{User-specific recommendation.}
Commercial recommendation is one of the most important applications of factorization models \cite{koren2009matrix,karatzoglou2010multiverse}. Here we illustrate how factor matrices are used for personalized recommendations with a real example. Figure \ref{fig:personal recommendation} shows the process for recommendation.
Below, we illustrate the process in detail.
\vspace*{-0.1cm}
\bit
\item An example user Tyler has a factor vector $\mathbf{u}$, namely user profile, which has been calculated by previous review histories.
\item We then calculate the personalized profile matrix $\T{R}=\T{G}\times_1\mathbf{u}(\in \mathbb{R}^{J_2 \times J_3})$. $\T{R}$ measures the amount of interaction of user profile with business and time factors.
\item Norm values of rows in $\T{R}$ indicate the influence of latent business concepts on Tyler. Dominant and weak concepts are found based on the calculated norm values. In the example, B4 is the strong, and B7 is the weak latent concept.
\item We inspect the corresponding columns of business factor matrix $\mathbf{U}^{(2)}$ and find relevant business entities with high values for the found concepts (B4 and B7).
\eit
\vspace*{-0.1cm}
We found both strong and weak entities by the above process. 
The strong and weak entities provide recommendation information by themselves in the sense that the probability of the user to like strong and weak entities are high and low, respectively, and they also give extended user preference information. For example, strong entities for Tyler are related to \lq spa \& health\rq~and located in neighborhood cities of Arizona, US. Weak entities are related to \lq grill \& restaurants\rq~and located in Toronto, Canada.
The captured user preference information makes commercial recommender systems more powerful with additional user-specific information such as address, current location, etc.
\begin{figure} [!t]
	\begin{center}
		\includegraphics[width=0.5 \textwidth]{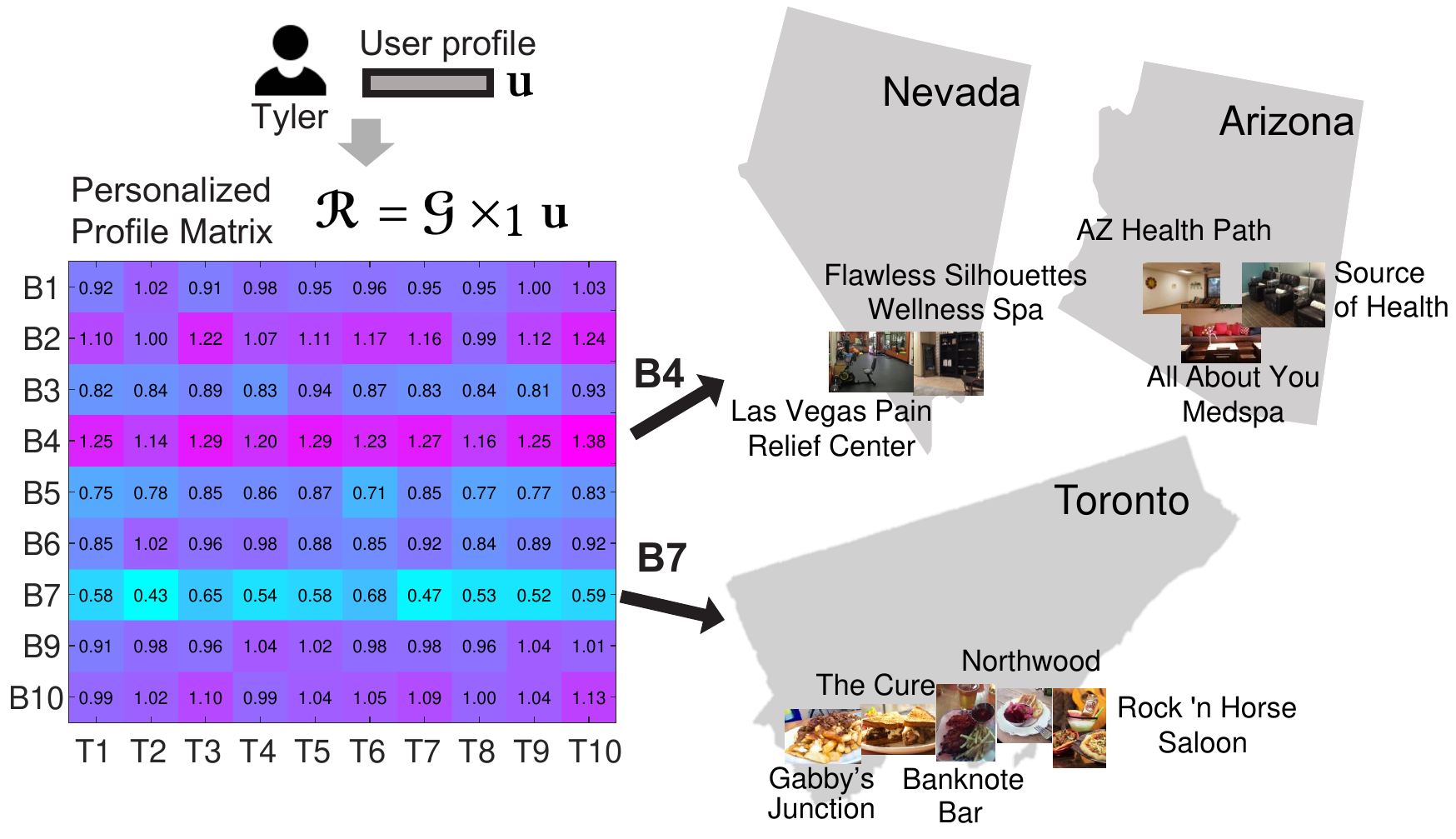}
	\end{center}
	\vspace*{-0.1cm}
	\caption{
		Example of personal recommendation process.}
	\label{fig:personal recommendation}
	\vspace*{-0.1cm}
\end{figure}
\vspace*{-0.2cm}

\section{Conclusion}
    \label{sec:conclusions}
    We propose \method, a fast, accurate, and scalable CMTF method.
\method significantly decreases the running time by lock-free parallel SGD update and reusing intermediate data. \method boosts up prediction accuracy by exploiting the sparsity of data, and inter-relations between factors.
\method shows 2.1$\sim$4.1$\times$ less error compared to the previous methods and improves the running time by 11$\sim$43$\times$.
\method shows linear scalability for the number of data entries and parallel cores.
Moreover, we show the usefulness of \method for cluster analysis and recommendation by applying \method to a real-world data Yelp. Future works include extending the method to a distributed setting.

\bibliographystyle{ACM-Reference-Format}
\bibliography{BIB/other}


\end{document}